\documentclass[11pt]{amsart}

\advance\hoffset by -0.5 truecm

\usepackage{amssymb}
\usepackage{amsthm}
\usepackage{amscd}

\usepackage{verbatim}
\usepackage{epsfig}
\usepackage{euscript}
\usepackage[colorlinks=true]{hyperref}
\hypersetup{urlcolor=blue, citecolor=red}
\usepackage{color}
\usepackage[shortlabels]{enumitem}
\usepackage{dsfont}

\newtheorem{Theorem}{Theorem}[section]
\newtheorem{Corollary}[Theorem]{Corollary}
\newtheorem{Lemma}[Theorem]{Lemma}
\newtheorem{Proposition}[Theorem]{Proposition}
\theoremstyle{definition}
\newtheorem{Definition}[Theorem]{Definition}
\newtheoremstyle{problemstyle}  
{3pt}                                               
{3pt}                                               
{\itshape}                               
{}                                                  
{\bfseries}                 
{\normalfont\bfseries:}         
{.5em}                                          
{}                                                  
\theoremstyle{problemstyle}

\theoremstyle{remark}
\newtheorem{cond}{Condition}[section]
\newtheorem{remark}{Remark}[section]

\def \R{\mathbb R}
\def \p{\partial}

\def \g{\gamma}

\def \<{\langle}
\def \>{\rangle}

\def \E{\mathbb{E}}

\def \e{\epsilon}
\def \ve{\varepsilon}

\def \l{\lambda}
\def \L{\Lambda}

\def \X{\mathcal{X}}
\def \D{\mathcal{D}}

\def \CM{\mathcal{C}^{3}_{\ell,M}(\Omega_0)}
\def \CMb{\mathcal{C}^{\beta}_{\ell,M}(\Omega_0)}
\def \H{\mathcal{H}}
\def \N{\mathcal{N}}
\def \M{\mathcal{M}}
\def \G{\Gamma}

\def \sub{\subseteq}
\def \gb{\bar{g}}
\def \O{\Omega}
\def \Oc{\overline{\Omega}}

\newcommand{\grad}{\operatorname{grad}}

\newcommand{\hess}{\operatorname{Hess}}
\newcommand{\supp}{\operatorname{supp}}

\newcommand{\diam}{\operatorname{diam}}
\newcommand{\dist}{\operatorname{dist}}

\newcommand{\vol}{\operatorname{Vol}}
\newcommand{\dom}{\operatorname{dom}}

\newcommand{\id}{\operatorname{id}}

\newcommand{\lsim}{\lesssim}

\def\id{\text{Id}}

\begin{document}

\title[Stability and statistical inversion of travel time tomography]{Stability and statistical inversion of travel time tomography}

\author[Ashwin Tarikere]{Ashwin Tarikere}
\address{Department of Mathematics, University of California Santa Barbara, Santa Barbara, CA 93106-3080, USA}
\email{ashwintan@ucsb.edu}

\author[Hanming Zhou]{Hanming Zhou}
\address{Department of Mathematics, University of California Santa Barbara, Santa Barbara, CA 93106-3080, USA}
\email{hzhou@math.ucsb.edu}


\begin{abstract}
In this paper, we consider the travel time tomography problem for conformal metrics on a bounded domain, which seeks to determine the conformal factor of the metric from the lengths of geodesics joining boundary points. We establish forward and inverse stability estimates for simple conformal metrics under some a priori conditions. We then apply the stability estimates to show the consistency of a Bayesian statistical inversion technique for travel time tomography with discrete, noisy measurements.
\end{abstract}
\maketitle

\section{Introduction}

Consider a smooth, bounded, and simply connected domain $\O \sub \R^m$, with $m\geq 2$. Given a Riemannian metric $g$ on $\Oc$, we define the associated \emph{boundary distance function} $\G_g:\p \O \times \p \O \to [0,\infty)$ by
\[
\G_g(\xi,\eta) = \inf \left\{ \int_\gamma \,d|g|:=\int_0^T |\dot\gamma(t)|_g\,dt \ : \ \gamma \in C^1([0,T],\Oc), \ \gamma(0)=\xi,\ \gamma(T)=\eta \right\}, \] 
for all $\xi,\eta \in \p \O$. In other words, $\G_g(\xi,\eta)$ is the Riemannian distance (with respect to $g$) between the boundary points $\xi$ and $\eta$. We consider the following inverse problem:
{\it 
	Can we recover the metric $g$ in the interior of the domain from the boundary distance function $\G_g$?
}

This inverse problem, called the {\it boundary rigidity problem} in mathematics literature, arose in geophysics in an attempt to determine the inner structure of the earth, such as the sound speed or index of refraction, from measurements of travel times of seismic waves on the earth's surface. This is called the {\it inverse kinematic problem} or the {\it travel time tomography problem} in seismology \cite{He1905,WZ1907}. 

The boundary rigidity problem is not solvable in general. Consider, for example, a unit disk with a metric whose magnitude is large (and therefore, geodesic speed is low) near the center of the disk. In such cases, it is possible that all distance minimizing geodesics connecting boundary points avoid the large metric region, and therefore one can not expect to recover the metric in this region from the boundary distance function. In view of this restriction, one needs to impose additional geometric conditions on the metric to be reconstructed. One such condition is {\it simplicity}. A metric $g$ on $\Oc$ is said to be simple if the boundary $\p \O$ is strictly convex w.r.t.\ to $g$ and any two points on $\Oc$ can be joined
by a unique distance minimizing geodesic. Michel conjectured that simple metrics are boundary distance rigid \cite{Mi81}, and this has been proved in dimension two \cite{PU05}. In dimensions $\geq 3$, this is known for generic simple metrics \cite{SU05}. When caustics appear, a completely new approach was established in \cite{SUV16, SUV21} for the boundary rigidity problem in dimensions $\geq 3$, assuming a convex foliation condition. Boundary rigidity problems for more general dynamical systems can be found in \cite{DPSU07, AZ15, Zhou18, PUZ19, LOY16, YUZ21, Plamen22}. We also refer to \cite{Croke04, SUVZ19} for summaries of recent developments on the boundary rigidity problem.

The boundary rigidity problem for general Riemannian metrics has a natural gauge: isometries of $(\Oc,g)$ that preserve $\p \O$ will also preserve the boundary distance function. In this paper, we restrict our attention to the problem of determining metrics from a fixed conformal class. Let $\bar g$ be a fixed ``background" metric on $\Oc$ which is simple and has $C^3$ regularity. For any positive function $n \in C^3(\Oc)$, define
$$g_n:=n^2\bar g,$$
which is a new Riemannian metric on  $\Oc$ that is conformal to $\bar g$. Our goal is to recover the parameter $n$ from the boundary distance function of $g_n$. In this problem, the gauge of isometries does not appear, and one expects to be able to uniquely determine the conformal factor $n$ from $\Gamma_{g_n}$. 

It is known that simple metrics from the same conformal class are boundary rigid for all $m\geq 2$ \cite{Mu75a,Mu75b,MR78}. To be precise, if $n_1,n_2 \in C^3(\Oc)$ are such that $g_{n_1},g_{n_2}$ are both simple metrics on $\Oc$, then $\G_{g_{n_1}}=\G_{g_{n_2}}$ if and only if $n_1=n_2$. To simplify notation, we will henceforth denote $\G_{g_n}$ by simply $\G_n$.

\subsection{Stability estimates for the deterministic inverse problem}

The uniqueness aspect of the boundary rigidity problem for conformal simple metrics has been quite well understood through the aforementioned studies \cite{Mu75a,Mu75b,MR78}. The first topic of this paper is the \emph{stability} of the boundary rigidity problem, i.e., quantitative lower bounds on the change in $\G_n$ corresponding to a change in the parameter $n$. Stability is important in practice, as we hope the inversion method for travel time tomography will be stable under perturbations of the data, e.g., by noise.

Conditional stability estimates for simple metrics can be found in \cite{Wa99, SU05, SUV16}, where the metrics are  assumed \emph{a priori} to be close to a given one. When considering a fixed conformal class, various stability estimates without the closeness assumption have been established in \cite{Mu75b, Mu81, Be79}. In \cite{Mu75b} the following stability result has been proved for the 2D boundary rigidity problem with the Euclidean background metric:
\begin{equation}\label{2D nonlinear stability}
\|n_1-n_2\|_{L^2(\O)}\leq \frac{1}{\sqrt{2\pi}}\|d_\xi (\Gamma_{n_1}-\Gamma_{n_2})(\xi,\eta)\|_{L^2(\p \O\times\p \O)}.
\end{equation}
Here, $d_\xi$ is the exterior derivative operator with respect to $\xi$ and the $L^2$ norms are taken with respect to the standard Euclidean metric. Notice that since the boundary distance function is symmetric, this estimate essentially says that the $L^2$-norm of $n_1-n_2$ can be controlled by the $H^1$-norm of $\Gamma_{n_1}-\Gamma_{n_2}$. For dimensions $\geq 3$, there are generalizations \cite{Be79, Mu81} of \eqref{2D nonlinear stability} with more complicated expressions (see also Theorem \ref{beylkin}). However, the estimates of \cite{Be79, Mu81} are not in standard Sobolev or H\"older norms, which makes them inconvenient for applications.

In this paper, we establish stability estimates similar to \eqref{2D nonlinear stability} for all dimensions $\geq 2$, without any \textit{a priori} closeness assumptions on $n_1,n_2$. Before giving the statement of our results, we need to define some function spaces for the conformal parameter $n$.

\begin{Definition}\label{nclass}
	Let $\O_0$ be a smooth, relatively compact subdomain of $\O$, and let $\l,\L, \ell, L$ be real numbers such that $$0< \l < 1 < \L, \qquad 0<\ell<L. $$
	We define $\N_{\l,\L,\ell,L}(\O_0)$ to be the set of all functions $n \in C^3(\Oc)$ that satisfy the following conditions:
	\begin{enumerate}[(i)]
		\item The metric $g_n = n^2\gb$ is a simple metric on $\Oc$.
		\item $\l < n(x) < \L$ for all $x \in \Oc$ and $n \equiv 1$ on $\Oc \setminus \O_0$.
		\item Let $\exp_n(x,v)$ denote the exponential map with respect to $g_n$ based at $x \in \Oc$ and acting on $v \in T_x\Oc$ (that is, the tangent space of $\Oc$ at $x$). Then the derivative of $\exp_n(x,\cdot)$ satisfies
		\begin{equation}\label{operatornorm0}
		   \ell|w|_{\gb} < |D_v \exp_n(x,v)(w)|_{\gb} < L|w|_{\gb}, 
		\end{equation}
		for all $x \in \Oc$, $v\in \dom(\exp_n(x,\cdot))$, and $w\in T_{v}T_x\Oc \cong T_x\Oc$.
	\end{enumerate}
	We also let
	\[ \N_{\l, \ell}(\O_0) := \bigcup_{\L>1, \, L>0}\N_{\l,\L,\ell,L}(\O_0). \]
\end{Definition}

The class of metrics associated with these function spaces includes any metric with non-positive sectional curvature that is conformal to $\gb$ and equal to $\gb$ in a neighborhood of $\p\O$ . Indeed, suppose $g_n = n^2\gb$ is such a metric. Then $(\Oc,g_n)$ is free of conjugate points by the curvature assumption, and $\p\O$ remains strictly convex with respect to $g_n$ since $g_n \equiv \gb$ near $\p\O$. Therefore, $g_n$ is a simple metric. Moreover, it follows from the Rauch Comparison Theorem that its exponential map $\exp_n$ satisfies \eqref{operatornorm0} for sufficiently large $L$ and any $\ell<1$  (see, e.g., \cite[Corollary 1.35]{CE08}).

\begin{remark}[Notation]
    Let $T:W_1 \to W_2$ be a linear map between normed vector spaces. Given real numbers $m, M$, we will use the notation
    $$m \prec T \prec M$$
    as shorthand for
    $$m\|w\|_{W_1} < \|Tw\|_{W_2} < M\|w\|_{W_1},$$
    for all $w \in W_1$. Using this notation, \eqref{operatornorm0} can be rewritten as
    \begin{equation}\label{operatornorm}
        \ell \prec D_v \exp_n(x,v) \prec L.
    \end{equation}
    We will also use $\|T\|_{op}$ to denote the \emph{operator norm} of $T$:
    $$\|T\|_{op} := \sup \left\{ \|Tw\|_{W_2} \ : \ w \in W_1, \  \|w\|_{W_1} =1\right\}.$$
    
\end{remark}

\begin{remark}\label{d}
    Let $\delta>0$ be the distance (w.r.t.\ to $\gb$) between $\p \O$ and $\overline{\O}_0$, and let $\xi,\eta \in \p\O$ be any pair of boundary points such that $\dist_{\gb}(\xi,\eta)<\delta$. For any $n \in \N_{\l,\ell}(\O_0)$, $g_{n}$ coincides with $\gb$ on $\Oc\setminus \O_0$, and consequently, we have $\G_n(\xi,\eta)= \dist_{\gb}(\xi,\eta)$. In particular, $\G_{n_1}(\xi,\eta)=\G_{n_2}(\xi,\eta)$ for all $n_1,n_2 \in \N_{\l,\ell}(\O_0)$.
\end{remark}

We are now ready to state our results on stability estimates for the boundary rigidity problem. The following ``inverse stability" estimate follows from a result of Beylkin \cite{Be79}, combined with some estimates for metrics with conformal factors $n \in \N_{\l,\ell}(\O_0)$. The details are presented in Section 2.

\begin{Theorem}\label{inverse}
	Let $\O,\O_0,\gb$ be as before, and let $\l,\ell$ be real numbers such that
    $$ 0 < \l < 1, \qquad 0 < \ell. $$
    Then there exists a constant $C_1(\O,\O_0,\gb,\ell)>0$ such that for all $n_1,n_2 \in \N_{\l,\ell}(\O_0)$,
	\[
	\|n_1-n_2\|_{L^2(\O)} \leq C_1\l^{2-m}\|d_{\xi}(\Gamma_{n_1}-\G_{n_2})(\xi,\eta)\|_{L^2(\p \O \times \p \O)}. \]
\end{Theorem}
Here, the $L^2$ norms are taken with respect to the background metric $\gb$, and $d_\xi$ represents the exterior derivative operator with respect to $\xi$. Please note that the stability constant $C_1$ can blow up as $\ell \to 0$. In a sense, as $\ell$ approaches $0$, we allow the metrics in our class to get closer and closer to potentially having conjugate points, and thus becoming non-simple.

We will apply the above stability estimate to study a statistical inversion technique for travel time tomography. For this purpose, we also need the following continuity (or ``forward stability") estimate of $\Gamma_n$. To the best of our knowledge, no such continuity estimate has been published before. The key idea in the proof is to apply the change of variables formula and use the upper bounds on $\det \left(D_v\exp_{n_j}\right)$ to control $\|\Gamma_{n_1}-\Gamma_{n_2}\|_{L^2}$ in terms of $\|n_1-n_2\|_{L^2}$.

\begin{Theorem}\label{forward}
    Let $\O,\O_0,\gb$ be as before, and let $\l,\L,\ell,L$ be real numbers such that
    $$ 0 < \l < 1 < \L, \qquad 0 < \ell < L. $$
	Then there exists a constant $C_2(\O,\O_0,\gb,\ell,L)>0$ such that for all $n_1,n_2 \in \N_{\l,\L,\ell,L}(\O_0)$,
	\[ \|\G_{n_1}-\G_{n_2}\|_{L^2(\p \O \times \p \O)} \leq C_2 \frac{\L^{m/2}}{\l}\|n_1-n_2\|_{L^2(\O)}. \]
\end{Theorem}

As with Theorem \ref{inverse}, the constant $C_2$ can blow up as $\ell \to 0$. The same happens as $L \to \infty$, since this allows $\det \left(D_v\exp_{n_j}\right)$ to blow up. The details are again postponed to Section 2.

\subsection{The statistical inverse problem}

The boundary rigidity problem is nonlinear, and geodesics are curved in general, so it is hard to derive explicit inversion formulas. Some reconstruction algorithms and numerical implementations based on theoretical analyses can be found in \cite{CQUZ07, CQUZ08, ACU19}. Typically, inversion methods in travel time tomography take an optimization approach with appropriate regularization. This is a deterministic approach which seeks to minimize some mismatch functional that quantifies the difference between the observations and the forecasts (synthetic data). However, this approach generally does not work well for non-convex problems. Moreover, various approximations in numerical methods can introduce systematic (random) error to the reconstruction procedure.

In this paper, we apply the above stability estimates (Theorems \ref{inverse} and \ref{forward})  to study a Bayesian inversion technique for the travel time tomography problem. The Bayesian inversion technique provides a reasonable solution for ill-posed inverse problems when the number of available observations is limited, which is a common scenario in practice. Applications of Bayesian inversion to seismology can be found in \cite{MWBG12, TGMS13}, which are based on the general paradigm of infinite dimensional Bayesian inverse problems developed by Stuart \cite{Stuart10}. However, most studies in the literature are concerned with waveform inversion, which is more PDE-based. On the other hand, there are very few results on statistical guarantees for the Bayesian approach to seismic inverse problems. These motivate us to apply Stuart's Bayesian inversion framework to produce a rigorous statistical analysis of the problem of recovering the wave speed from the (noisy) travel time measurements.

For statistical inversion, it is convenient to rewrite the conformal factor $n$ using an exponential parameter: For any $\beta \geq 3$, let $C_0^\beta(\O_0)$ denote the closure in the H\"older space $C^{\lfloor \beta\rfloor, \beta - \lfloor \beta \rfloor}(\overline{\O}_0)$ of the subspace of all smooth functions compactly supported in $\O_0$. Given any function $c\in C_0^3(\O_0)$, we define the corresponding conformal factor $n_c$ by 
\begin{equation}\label{ncdef}
      n_c(x) = \begin{cases}
    e^{c(x)} & \textrm{if } x \in \O_0, \\
    1 & \textrm{if } x \in \Oc \setminus \O_0.
\end{cases}
\end{equation}

It is easy to see that $n_c$ is a positive $C^3$ function on $\Oc$. To simplify notation, we will denote the corresponding boundary distance function $\G_{n_c}$ by simply $\G_c$. 

Our goal is to reconstruct the exponential parameter $c$ from error-prone measurements of $\Gamma_c$ on finitely many pairs of boundary points $(X_i,Y_i)$, $i = 1, \ldots, N$. Following the general paradigm of Bayesian inverse problems, we assume that $c$ arises from a prior probability distribution $\Pi$ on $C^3_0(\O_0)$.  We will construct $\Pi$ so that it is supported in a subset of $C_0^3(\O_0)$ of the following form:

\begin{Definition}\label{CMb}
    Let $\ell,M>0$ and $\beta \geq 3$. We define $\CMb$ as the set of all functions $c \in C_0^\beta(\O_0)$ that satisfy the following conditions:
    \begin{enumerate}[(i)]
        \item The metric $g_{n_c}=n_c^{2}\gb$ is a simple metric on $\Oc$.
        \item The derivative of $\exp_{n_c}(x,\cdot)$ satisfies
        $$D_w \exp_{n_c}(x,w) \succ \ell,$$
        for all $x\in \Oc$ and $w\in \dom(\exp_{n_c}(x,\cdot))$.
        \item $\|c\|_{C^{\lfloor \beta\rfloor, \beta - \lfloor \beta \rfloor}(\overline{\O}_0)} < M$.
    \end{enumerate}
\end{Definition}
 We will show in Section 2 that if $c \in \CMb$, the corresponding conformal parameter $n_c \in \N_{\l,\L,\ell,L}(\O_0)$ for appropriate choices of $\l,\L$ and $L$. The precise construction of $\Pi$ is described in Section 3.

\begin{remark}[Notation]
    Henceforth, we will denote $C^{\lfloor \beta\rfloor, \beta - \lfloor \beta \rfloor}$ by simply $C^\beta$.
\end{remark}

\begin{remark}\label{open}
    It is known that small perturbations of simple metrics are again simple. Therefore, $\CMb$ is an open subset of $C^\beta_0(\O_0)$. 
\end{remark}

The pairs of boundary points $(X_i,Y_i)$ between which the distance measurements are to be made are chosen according to the rule
\[
(X_i,Y_i) \stackrel{\textrm{i.i.d.}}{\sim} \mu, \]
where $\mu$ is the uniform probability measure on $\p \O \times \p \O$ induced by the background metric $\bar g$. The actual distance measurements between these points are assumed to be of the form
\[
\Gamma_i = e^{\e_i}\Gamma_c(X_i,Y_i),  \]
where $\e_i$ are i.i.d. $N(0, \sigma^2)$ normal random variables ($\sigma >0$ is fixed) that are also independent of $(X_j,Y_j)_{j=1}^N$. For simplicity, we will henceforth assume that $\sigma = 1$ without loss of generality. Define
\[
Z_c = \log \Gamma_c, \]
and for $i= 1, \ldots, N$, 
\begin{align*}
	Z_i &= \log \Gamma_i\\
	&= Z_c(X_i,Y_i) +\e_i.
\end{align*}
All of our measurements can be summarized using the data vector 
\begin{equation} \label{obs}
	\mathcal{D}_N = (X_i,Y_i,Z_i)_{i=1}^N \in (\p \O \times \p \O \times \R)^N. 
\end{equation}
For convenience, let us define $\X = \p \O \times \p \O \times \R$. 

Next, let $P^N_c$ denote the probability law of $\mathcal{D}_N|c$. It is easy to see that $P^N_c = \times_{i=1}^N P^{(i)}_c$, where for each $i \in \{1, \ldots , N\}$, $P^{(i)}_c$ is equal to the probability law of $(X_i,Y_i,Z_i)$. More explicitly, for each $i\in \{1,\ldots , N\}$,
\[
dP^{(i)}_c(x,y,z) = p_cd\mu(x,y)dz, \]
where
$$ p_c(x,y,z) = \frac{1}{\sqrt{2\pi}}\exp\left\{-\frac{1}{2}\left(z-Z_c(x,y)\right)^2\right\}.$$

We denote the posterior distribution of $c|\mathcal{D}_N$ by $\Pi(\cdot |\mathcal{D}_N)$. By Corollary \ref{forward-c}, the map $(c,(x,y,z))\mapsto p_c(x,y,z)$ is jointly Borel-measurable from $C^3_0(\O_0)\times \X$ to $\R$. So it follows from standard arguments (see \cite[p.~7]{GvdV17} ) that the posterior distribution is well-defined and takes the form
\[
\Pi(A|\D_N) = \frac{\int_A \prod_{i=1}^N p_c(X_i,Y_i,Z_i)d\Pi(c)}{\int \prod_{i=1}^N p_c(X_i,Y_i,Z_i) d\Pi(c)} \]
for any Borel set $A \subseteq C^3_0(\O_0)$. Our posterior estimator for $c$ will be the posterior mean
\begin{equation}\label{postmean}
	\overline{c}_N = \E^{\Pi}[c|\mathcal{D}_N]. 
\end{equation}

\begin{Theorem}\label{mainth}
  Suppose that the true parameter $c_0$ is smooth and compactly supported in $\O_0$, and is such that $g_{n_{c_0}}$ is a simple metric on $\Oc$. Then there is a well defined prior distribution $\Pi$ on $C^3_0(\O_0)$ such that the posterior mean $\overline{c}_N$ satisfies
	\[
	\|\overline{c}_N - c_0\|_{L^2(\O)} \to 0  \]
    in $P^N_{c_0}$- probability, as $N \to \infty$.
\end{Theorem}

A more precise version of this result is stated in Theorem \ref{main} in Section \ref{stats}, which in fact requires significantly weaker regularity assumptions on $c_0$. It also specifies an explicit $N^{-\omega}$ rate of convergence, where $\omega$ is a positive constant that can be made arbitrarily close to $1/4$.

To prove Theorem \ref{mainth}, we apply the analytic techniques developed in recent consistency studies of statistical inversion of the geodesic X-ray transform \cite{MNP19} and related non-linear problem arising in polarimetric neutron tomography \cite{MNP21a, MNP21b}. The forward and inverse stability estimates for the measurement operators (like the ones in Theorems \ref{inverse} and \ref{forward}) play a key role in the arguments of these references. 

The analysis of theoretical guarantees for statistical inverse problems is currently a very active topic. Recent progress for various linear and non-linear inverse problems include \cite{DaS16, DuS16, AN19, MNP19, Nickl20, MNP21a, MNP21b, NP21, BN21, Bohr22}. See also the recent lecture notes \cite{Nickl22}.
\bigskip

The paper is structured as follows. In Section 2, we establish the forward and inverse stability estimates for the boundary distance function. Section 3 is devoted to proving the statistical consistency of Bayesian inversion for the boundary rigidity problem.


\section{Forward and Inverse continuity estimates}

In order to prove the statistical consistency of the proposed Bayesian estimator, we need to establish quantitative upper and lower bounds on the magnitude of change in the boundary distance function $\G_n$ corresponding to a change in the conformal parameter $n$ of the metric. This is the content of Theorems \ref{inverse} and \ref{forward}, which we will prove in this section. We will also use these estimates to establish similar bounds for the map $c \mapsto Z_c = \log \G_c$, when $c$ belongs to the parameter space $\CMb$ defined in Definition \ref{CMb}.

\subsection{Stability estimates}

We begin with the proof of Theorem \ref{inverse}. As we noted in the introduction, such an estimate has already been proved for dimension $m=2$ by Mukhometov in \cite{Mu75b}. For general $m\geq 2$, we have the following result by Beylkin \cite{Be79}. Also see \cite[Lemma 4]{Mu81}.

\begin{Theorem}[{\cite{Be79}}]\label{beylkin} Let $n_1,n_2 \in C^3(\Oc)$ be such that $g_{n_1},g_{n_2}$ are simple metrics on $\Oc$. Then 
	\begin{equation}\label{hds}
		\begin{split}
			\int_{\O} (n_1-n_2) & (n_1^{m-1}-n_2^{m-1})\,d\textrm{Vol}_{\gb}\\
			& \leq C_m \int_{\p \O_{\xi}\times\p \O_{\eta}} \sum_{a+b=m-2} d_\xi (\Gamma_{n_1}-\Gamma_{n_2})\wedge d_\eta (\Gamma_{n_1}-\Gamma_{n_2})\wedge (d_\xi d_\eta \Gamma_{n_1})^a \wedge (d_\xi d_\eta \Gamma_{n_2})^b\, ,
		\end{split}
	\end{equation}
	where $d\textrm{Vol}_{\gb}$ is the Riemannian volume form induced by $\gb$, and $d_\xi$ and $d_\eta$ represent the exterior derivative operators on $\p \O$ with respect to $\xi$ and $\eta$ respectively. Given local coordinates $(\xi^1, \ldots , \xi^{m-1})$ for $\xi$ and $(\eta^1, \ldots , \eta^{m-1})$ for $\eta$, we have $d_\xi = d\xi^i \frac{\p}{\p \xi^i}$, $d_\eta = d\eta^j\frac{\p}{\p \eta^j}$, and $d_\xi d_\eta=d\xi^i\wedge d\eta^j \frac{\p^2}{\p \xi^i \p \eta^j}$. The constant 
	$$C_m=\frac{(-1)^{\frac{(m-1)(m-2)}{2}}\Gamma(m/2)}{2\pi^{m/2}(m-1)!},$$
	depends only on the dimension $m$. 
\end{Theorem}

We will show that when $n_1,n_2 \in \N_{\l,\ell}(\O_0)$, the inequality \eqref{hds} leads to the desired stability estimate.

\begin{Lemma}\label{derivative estimates}
	Let $n \in \N_{\l,\ell}(\O_0)$. Then the corresponding boundary distance function $\G_n$ satisfies 
	$$ |d_{\xi}\G_n(\xi,\eta)|_{\gb} \leq 1, \quad |d_{\eta}\G_n(\xi,\eta)|_{\gb} \leq 1, $$
	and
	$$|\nabla^{\xi}\nabla^{\eta}\G_n(\xi,\eta)|_{\gb} \leq \frac{(1+\ell^{-1})}{\l}\dist_{\gb}(\xi,\eta)^{-1} $$
	for all $\xi,\eta \in \p \O$ with $\xi \neq \eta$. Here, $\nabla^\xi, \nabla^\eta$ denote the covariant derivative operators with respect to $\xi$ and $\eta$ respectively, and $\dist_{\gb}(\xi,\eta)$ is the distance from $\xi$ to $\eta$ with respect to the metric $\gb$.
\end{Lemma}
\begin{proof}
	Given $\xi,\eta \in \p \O$ with $\xi \neq \eta$, let $v(\xi,\eta)$ denote the unit vector (with respect to $g_n$) at $\eta$ tangent to the geodesic from $\xi$ to $\eta$. It follows from the First Variation Formula (cf. \cite{Lee-RM}, Theorem 6.3) that the gradient (with respect to $g_n$) of $\G_n(\xi,\cdot)$ is given by
	\begin{equation}\label{gradeta}
		\grad_{\eta}\G_n(\xi,\eta) = \Pi_{\eta} v(\xi,\eta),
	\end{equation} 
	where $\Pi_{\eta} : T_\eta \Oc \to T_\eta \p \O$ is the orthogonal projection map onto the tangent space of the boundary. Since $g_n = \gb$ on $\p \O$, it follows immediately that 
	$$|d_{\eta}\G_n(\xi,\eta)|_{\gb} = |\grad_{\eta}\G_n(\xi,\eta)|_{g_n} = \left|\Pi_\eta v(\xi,\eta)\right|_{g_n} \leq |v(\xi,\eta)|_{g_n}=1. $$
	Similar arguments show that $|d_{\xi}\G_n(\xi,\eta)|_{\gb}\leq 1$ as well. 
	
	Next, let $(\xi^1, \ldots , \xi^{m-1})$ and $(\eta^1, \ldots , \eta^{m-1})$ be local coordinates for $\p \O$ around $\xi$ and $\eta$ respectively. We can extend these coordinate charts to boundary normal coordinates $(\xi^1, \ldots , \xi^m)$ and $(\eta^1, \ldots , \eta^m)$ by taking $\xi^m$ and $\eta^m$ to be the corresponding distance functions from the boundary. With respect to these coordinates, we may rewrite \eqref{gradeta} as
	\begin{equation}\label{gradetac}
	    \grad_\eta \G_n(\xi,\eta) = \sum_{j=1}^{m-1} v^j(\xi,\eta)\frac{\p}{\p \eta^j}.
	\end{equation}
	We can extend both sides of this equality to $(1,0)$-tensor fields on $\p \O_\xi \times \p \O_\eta$, while maintaining the equality. Taking covariant derivatives of both sides with respect to $\xi$, we get
	\begin{equation}\label{dxigradetac}
	    \nabla^\xi \grad_\eta \G_n(\xi,\eta) = \sum_{i,j=1}^{m-1}\frac{\p v^j}{\p \xi^i}(\xi,\eta) \frac{\p}{\p \eta^j}\otimes d\xi^i.
	\end{equation}
    Here, we have used the fact that the product connection on $\p \O_{\xi}\times \p \O_{\eta}$ satisfies $\nabla_{\p_{\xi_i}}\p_{\eta_j}=0$ for all $i,j$. Recall that $g_n$ is a simple metric, and its exponential map $\exp_n(x,\cdot)$ at any $x \in \Oc$ is a diffeomorphism onto $\Oc$. Let $w(x,\cdot):\Oc \to T_x\Oc$ denote its inverse map. Since $D_v \exp_n(x,v)\succ \ell$ for all $v$ in the domain of $\exp_n(x,\cdot)$, we have
	\begin{equation}\label{dwnorm}
	    \|D_y w(x,y)\|_{op} < \ell^{-1} \qquad \textrm{for all } y \in \Oc.
	\end{equation}
	Now observe that we have the identity
	$$ v(\xi,\eta) = -\frac{w(\eta, \xi)}{\G_n(\xi,\eta)}. $$
	So by \eqref{gradetac} and \eqref{dxigradetac}, 
	\begin{align}
		\nabla^\xi \grad_{\eta}\G_n(\xi,\eta) &=-\sum_{i,j=1}^{m-1}\left\{\frac{1}{\G_n(\xi,\eta)}\frac{\p w^j (\eta, \xi)}{\p \xi^i}-\frac{w^j(\eta,\xi)}{\G_n(\xi,\eta)^2}\frac{\p \G_n(\xi,\eta)}{\p \xi^i}\right\} \frac{\p}{\p \eta^j}\otimes d\xi^i \nonumber \\
		&= -\frac{1}{\G_n(\xi,\eta)}\left\{\sum_{i,j=1}^{m-1}\frac{\p w^j (\eta, \xi)}{\p \xi^i} \frac{\p}{\p \eta^j} \otimes d\xi^i\right\} + \frac{1}{\G_n(\xi,\eta)}v(\xi,\eta)\otimes d_\xi \G_n(\xi,\eta). \label{coordfree}
	\end{align}
	Observe that $\sum_{i,j=1}^{m-1}\frac{\p w^j (\eta, \xi)}{\p \xi^i} \frac{\p}{\p \eta^j} \otimes d\xi^i$ is precisely the tensor form of the linear map 
    $$\Pi_\eta \circ D_y w(\eta, y)\big|_{y=\xi}\circ \Pi_\xi,$$
    where $\Pi_\xi$ and $\Pi_\eta$ are, as before, orthogonal projections from $T_\xi \Oc \to T_\xi \p \O$ and $T_\eta \Oc \to T_\eta \p \O$ respectively. Therefore,
	$$ \left|\sum_{i,j=1}^{m-1}\frac{\p w^j (\eta, \xi)}{\p \xi^i} \frac{\p}{\p \eta^j} \otimes d\xi^i\right|_{\gb} \leq \left\|D_y w(\eta, y)\big|_{y=\xi}\right\|_{op} < \ell^{-1}. $$
	Combining this with \eqref{coordfree}, we get
	\begin{align*}
	    |\nabla^\xi d_\eta \G_n(\xi,\eta)|_{\gb} &= |\nabla^\xi \grad_\eta \G_n(\xi,\eta)|_{\gb} \\
	    &\leq \frac{\ell^{-1}}{\G_n(\xi,\eta)} + \frac{|v(\xi,\eta)|_{\gb}|d_\xi \G_n(\xi,\eta)|_{\gb}}{\G_n(\xi,\eta)} \\
	    &\leq \frac{(1+\ell^{-1})}{\G_n(\xi,\eta)}.
	\end{align*}
	Finally, applying the simple estimate
	$$ \dist_{\gb}(\xi,\eta) \leq \frac{1}{\lambda}\G_n(\xi,\eta), $$
	we get
	$$|\nabla^\xi \nabla^\eta\G_n(\xi,\eta)|_{\gb} = |\nabla^\xi d_\eta\G_n(\xi,\eta)|_{\gb} \leq \frac{(1+\ell^{-1})}{\lambda}\dist_{\gb}(\xi,\eta)^{-1}.$$
	This completes the proof. 
\end{proof}

With these estimates in hand, we're now ready to prove Theorem \ref{inverse}.

\begin{proof}[Proof of Theorem \ref{inverse}]
	Consider the inequality \eqref{hds} from Theorem \ref{beylkin}. For $n_1,n_2 \in \N_{\l,\ell}(\O_0)$, the left hand side becomes
	\begin{equation}\label{lowboundn}
		\int_{\O} (n_1-n_2)^2(n_1^{m-2}+n_1^{m-3}n_2+\cdots +n_2^{m-2})d \textrm{Vol}_{\gb} \geq (m-1)\l^{m-2}\|n_1-n_2\|_{L^2(\O)}^2.
	\end{equation}
    Now consider the right hand side of \eqref{hds}. By Lemma \ref{derivative estimates},
    $$|d_\xi d_\eta \G_n|_{\gb} = \left|\textrm{Alt}\left(\nabla^\xi \nabla^\eta \G_n\right)\right|_{\gb} \leq \frac{(1+\ell^{-1})}{\l}\dist_{\gb}(\xi,\eta)^{-1}.$$
	Therefore, the right hand side of \eqref{hds} is bounded above by
	\begin{align*}
		& \phantom{\leq }|C_m|\int_{\p \O \times \p \O} |d_{\xi}(\G_{n_1}-\G_{n_2})|_{\gb}|d_\eta(\G_{n_1}-\G_{n_2})|_{\gb}\sum_{a+b=m-2}|d_{\xi}d_{\eta}\G_{n_1}|_{\gb}^a|d_{\xi}d_{\eta}\G_{n_2}|_{\gb}^b\,  d\sigma_{\gb} \\
		&\leq (m-1)|C_m|\frac{(1+\ell^{-1})^{m-2}}{\l^{m-2}} \int_{\p \O \times \p \O} |d_{\xi}(\G_{n_1}-\G_{n_2})|_{\gb}|d_\eta (\G_{n_1}-\G_{n_2})|_{\gb}|\dist_{\gb}(\xi,\eta)|^{2-m}\, d\sigma_{\gb},  
	\end{align*}
	where $d\sigma_{\gb}$ is the surface measure on $\p \O \times \p \O$ induced by $\gb$. Observe that by Remark \ref{d}, we have  $(\G_{n_1}-\G_{n_2})(\xi,\eta) =0$ for all $\xi,\eta \in \p \O$ with $\dist_{\gb}(\xi,\eta)<\delta$. Therefore, the above expression is further bounded above by
    \begin{align*}
        & \phantom{\leq } (m-1)|C_m|\frac{(1+\ell^{-1})^{m-2}}{\l^{m-2}}\delta^{2-m}\int_{\p \O \times \p \O} |d_{\xi}(\G_{n_1}-\G_{n_2})|_{\gb}|d_\eta (\G_{n_1}-\G_{n_2})|_{\gb}|d\sigma_{\gb}. \\
        &\lsim_{m,\delta, \ell}\lambda^{2-m}\left( \|d_{\xi}(\G_{n_1}-\G_{n_2})\|^2_{L^2(\p \O \times \p \O)} + \|d_{\eta}(\G_{n_1}-\G_{n_2})\|^2_{L^2(\p \O \times \p \O)}\right)\\
        & \lsim_{m,\delta, \ell} \lambda^{2-m} \|d_{\xi}(\G_{n_1}-\G_{n_2})\|^2_{L^2(\p \O \times \p \O)}
    \end{align*}
	since $\|d_\xi(\G_{n_1}-\G_{n_2})\|_{L^2} = \|d_\eta(\G_{n_1}-\G_{n_2})\|_{L^2}$ by symmetry. Combining this with \eqref{lowboundn}, we get
	$$\|n_1-n_2\|^2_{L^2(\O)} \lsim_{m,\delta,\ell} \lambda^{2(2-m)}\|d_{\xi}(\G_{n_1}-\G_{n_2})\|^2_{L^2(\p \O \times \p \O)} $$
	and the theorem follows.
\end{proof}

Recall that we parametrized the conformal parameter $n$ of the metric $g_n$ by a function $c$ belonging to the parameter space $\CMb$, as defined in \eqref{ncdef}. We assumed that our input data consists of finitely many measurements of the function $Z_c = \log \G_{c}$. In the following corollary, we translate Theorem \ref{inverse} into stability estimates for the map $c \mapsto Z_c$ using simple Lipschitz estimates for the exponential function: For all $x,y \in [M_1,M_2]$,
\begin{equation}\label{lipexp}
    e^{M_1}|x-y| \leq |e^x-e^y| \leq e^{M_2}|x-y|.
\end{equation}
This immediately implies that for all $c_1,c_2\in \CMb$,
\begin{equation}\label{link-L2}  
	  e^{-M}\|c_1-c_2\|_{L^2(\O_0)} \leq \|n_{c_1}-n_{c_2}\|_{L^2(\O)} \leq e^M\|c_1-c_2\|_{L^2(\O_0)}. 
\end{equation}  

\begin{Corollary}\label{inverse-c}
	For any $M>0$, there exists a constant $C_1' = C_1'(\O,\O_0,\gb,\ell,M)>0$ such that
	\[
	\|c_1-c_2\|_{L^2(\O_0)} \leq C_1'\|Z_{c_1}-Z_{c_2}\|_{H^1(\p \O \times \p \O)} \]
	for all $c_1,c_2 \in \CM$.
\end{Corollary}
\begin{proof}
	Let $c_1,c_2 \in \CM$. Then $n_{c_1}, n_{c_2} \in \N_{\l,\ell}(\O_0)$ for $\l = e^{-M}$. So it follows from Theorem \ref{inverse} that 
	\begin{equation}\label{c2.3temp}
		\|n_{c_1}-n_{c_2}\|_{L^2(\O)}\leq C_1e^{(m-2)M}\|d_\xi(\G_{c_1}-\G_{c_2})\|_{L^2(\p \O \times \p \O)}. 
	\end{equation}
    By \eqref{link-L2}, the left hand side of the above equation is bounded below by $e^{-M}\|c_1-c_2\|_{L^2(\O_0)}$. Now, rewrite $d_{\xi}(\G_{c_1}-\G_{c_2})$ as
	\begin{equation*}
		\begin{split}
			d_{\xi}(\G_{c_1}-\G_{c_2}) & = d_\xi(e^{Z_{c_1}}-e^{Z_{c_2}})\\
			& = e^{Z_{c_1}}d_{\xi}Z_{c_1}-e^{Z_{c_2}}d_{\xi}Z_{c_2} \\
			& = e^{Z_{c_1}}d_{\xi}(Z_{c_1}-Z_{c_2})+(e^{Z_{c_1}}-e^{Z_{c_2}})d_\xi Z_{c_2}. 
		\end{split}
	\end{equation*}
	It follows from Remark \ref{d} that if $(\xi,\eta) \in \supp(\G_{c_1}-\G_{c_2})$, we have $\dist_{\gb}(\xi,\eta)\geq \delta$, and consequently,
	\[
	e^{-M}\delta \leq \G_{c_j}(\xi,\eta) \leq e^M\diam_{\gb}(\O), \qquad j=1,2.
	\]
	Therefore, by applying \eqref{lipexp} along with the fact that $|d_\xi\G_{c_j}|_{\gb}\leq 1$ by Lemma \ref{derivative estimates}, we get
	\begin{equation*}
		\begin{split}
			|d_\xi(\G_{c_1}-\G_{c_2})|_{\gb} & \leq |\G_{c_1}||d_\xi(Z_{c_1}-Z_{c_2})|_{\gb}+|\G_{c_1}-\G_{c_2}||d_\xi\G_{c_2}|_{\gb}/|\G_{c_2}|\\
            &\leq e^M\diam_{\gb}(\O)|d_\xi(Z_{c_1}-Z_{c_2})|_{\gb}+\frac{|e^{Z_{c_1}}-e^{Z_{c_2}}|}{e^{-M}\delta} \\
			&\leq e^M\diam_{\gb}(\O)|d_\xi(Z_{c_1}-Z_{c_2})|_{\gb}+\frac{e^M\diam_{\gb}(\O)}{e^{-M}\delta}|Z_{c_1}-Z_{c_2}|, 
		\end{split}
	\end{equation*}
	where $\diam_{\gb}(\O)$ denotes the diameter of $\O$ with respect to the metric $\gb$. This further implies
	$$ \|d_{\xi}(\G_{c_1}-\G_{c_2})\|_{L^2(\p \O \times \p \O)}  \lsim_{\O,\gb, \delta, \ell, M} \|Z_{c_1}-Z_{c_2}\|_{H^1(\p \O \times \p \O)}. $$
	Combining this with \eqref{link-L2}  and \eqref{c2.3temp}, we get
	$$ \|c_1-c_2\|_{L^2(\O_0)} \lsim_{\O,\gb, \delta, \ell, M} \|Z_{c_1}-Z_{c_2}\|_{H^1(\p \O \times \p \O)}. $$
	This completes the proof.
\end{proof}

\subsection{Forward continuity estimates}
We now move on to the proof of Theorem \ref{forward}. The key idea is to use \emph{upper bounds} on $D_v \exp_{n_j}(x,v)$ to control $\|\G_{n_1}-\G_{n_2}\|_{L^2}$ with respect to $\|n_1-n_2\|_{L^2}$. 

We begin by introducing some notation. Let $S\Oc$ denote the unit sphere bundle on $\Oc$, that is,
$$S\Oc = \{(x,v) \in T\Oc \ : \ |v|_{\gb}=1 \}.$$
The boundary of $S\Oc$ consists of unit tangent vectors at $\p\O$. Specifically,
$$\p S\Oc = \{(x,v) \in S\Oc \ : \ x \in \p\O\}. $$
Let $\nu$ denote the inward unit normal vector field along $\p \O$ with respect to the metric $\gb$. We define the bundles of \emph{inward pointing} and \emph{outward pointing} unit tangent vectors on $\p\O$ as follows:
\begin{align*}
	\p_+S\Oc &:= \left\{ (\xi,v) \in \p S\Oc \ : \ \langle v, \nu_{\xi}\rangle_{\gb} \geq 0 \right\}, \quad \textrm{and } \\
	\p_{-}S\Oc &:= \left\{ (\xi,v) \in \p S\Oc \ : \ \ \langle v, \nu_{\xi}\rangle_{\gb} \leq 0 \right\}.
\end{align*}
We also set
$$\p_0S\Oc := \p_{+}S\Oc \cap \p_{-}S\Oc.$$
This coincides with $S\p\O$, the unit sphere bundle on $\p \O$.

Next, let $n \in N_{\l,\ell}(\O_0)$. For $(\xi,v) \in \p_{+}S\Oc$, we let $\g_n(\xi,v,t) = \exp_n(\xi,tv)$ denote the unit speed geodesic (with respect to $g_n$) starting at $\xi$ with initial direction $v$ at time $t=0$. We define $\tau_n(\xi,v)$ to be the time at which $\g_n(\xi,v,\cdot)$ exits $\Oc$. It is known (see \cite{GeomIP}) that for simple manifolds, $\tau_n$ is a $C^1$ function of $\p_{+}S\Oc$, and $\tau_n(\xi,v)=0$ if and only if $v \in S_\xi \p\O$. We also define $\eta_n(\xi,v)$ and $u_n(\xi,v)$ as the point and direction at which $\g_n(\xi,v,\cdot)$ exits $\Oc$. In other words,
\begin{align*}
   \eta_n(\xi,v) &:= \g_n(\xi,v,\tau_n(\xi,v)), \quad \textrm{and} \\
   u_n(\xi,v) &:= \dot{\g}_{n}(\xi,v,\tau_n(\xi,v)). 
\end{align*}

\begin{Lemma}\label{derivative of tau}
	Let $n \in \N_{\l,\L,\ell,L}(\O_0)$. Then for all $(\xi,v)\in \p_+S\Oc$,
	\[
	\|D_v\tau_n(\xi,v)\|_{op} \leq L\frac{\tau_n(\xi,v)}{\< \nu, u\>_{\gb}} \leq \frac{L\L \diam_{\gb}(\O)}{\<\nu,u\>_{\gb}}, \]
	where $\nu = \nu_{\eta_n(\xi,v)}$ and $u=u_n(\xi,v)$.
\end{Lemma}

\begin{proof}
	Let $\rho \in C^1(\Oc)$ be such that $\rho^{-1}(0) = \p \O$ and $\rho(x) = \dist_{\gb}(x,\p \O)$ for $x$ near $\p \O$. Consider the function
	$$ f(t,v) = \rho(\exp_n(\xi,tv)). $$
	Observe that
	$$\frac{\p f}{\p t}\Big|_{t=\tau_n(\xi,v)} = \left\langle (\grad \rho)_{\eta_n(\xi,v)}, u_n(\xi,v)\right\rangle_{\gb} = \< \nu, u\>_{\gb}.$$
	On the other hand,
	\begin{align*}
		D_vf(t,v) &= D\rho_{\exp_n(\xi,tv)}\circ \left(tD_w\exp_n(\xi,w)\big|_{w=tv}\right) \\
		\Rightarrow D_vf\big|_{(\tau_n(\xi,v),v)} &= \tau_n(\xi,v)\Pi^{\nu}\circ D_w\exp_n(\xi,w)\big|_{w = \tau_n(\xi,v)v},
	\end{align*}
	where $\Pi^{\nu}$ is the linear map given by
	$$\Pi^{\nu}(w) = \langle \nu, w\rangle_{\gb} \qquad \textrm{for all } w \in T_{\eta_n(\xi,v)}\Oc.$$
	Now differentiating the identity $f(\tau_n(\xi,v),v)=0$ with respect to $v$, we get
	\begin{equation*}
		\begin{split}
			0 &= \frac{\p f}{\p t}\Big|_{(\tau_n(\xi,v),v)}D_v\tau_n(\xi,v)+D_vf\big|_{(\tau_n(\xi,v),v)} \\
			&= \<\nu,u\>_{\gb} D_v\tau_n(\xi,v)+\tau_n(\xi,v)\Pi^\nu \circ D_w\exp_n(\xi,w)\big|_{w=\tau_n(\xi,v)v}.
		\end{split}
	\end{equation*}
	Therefore,
	\begin{align*}
		D_v\tau_n(\xi,v) &= -\frac{\tau_n(\xi,v)}{\<\nu,u\>_{\gb}}\Pi^\nu \circ D_w\exp_n(\xi, w)\big|_{w=\tau_n(\xi,v)v} \\
		\Rightarrow \|D_v\tau_n(\xi,v)\|_{op} &\leq \frac{\tau_n(\xi,v)}{\<\nu,u\>_{\gb}}\left\|D_w\exp_n(\xi,w)\big|_{w=\tau_n(\xi,v)v}\right\|_{op} \\
		&\leq L\left[\frac{\tau_n(\xi,v)}{\<\nu,u\>_{\gb}}\right],
	\end{align*}
	as required. Now the lemma follows by observing that 
	$$\tau_{n}(\xi,v) \leq \diam_{g_n}(\O) \leq \L\diam_{\gb}(\O),$$
	for all $(\xi,v)\in \p_+S\Oc$.
\end{proof}

We are now ready to prove Theorem \ref{forward}. Recall that the notation $\int_\g f d|g|$ denotes the integral of a function $f$ along the curve $\g$ with respect to the arc-length metric induced by $g$.
\begin{proof}[Proof of Theorem \ref{forward}]
	Fix $\xi \in \p \O$, and define the sets
	\begin{align*}
		B_1(\xi) &:= \{\eta \in \p \O \ : \ \G_{n_1}(\xi,\eta) \leq \G_{n_2}(\xi,\eta)\},\\
		B_2(\xi) &:= \{\eta \in \p \O \ : \ \G_{n_2}(\xi,\eta) \leq \G_{n_1}(\xi,\eta)\}.
	\end{align*}Suppose $\eta \in B_1(\xi)$, and let $\g_1(\xi,\eta)$ denote the unit speed geodesic with respect to $g_{n_1}$ from $\xi$ to $\eta$. Clearly, $\G_{n_1}(\xi,\eta) = \int_{\g_1(\xi,\eta)}n_1d|\gb|$, whereas $\G_{n_2}(\xi,\eta) \leq \int_{\g_1(\xi,\eta)}n_2 d|\gb|$. So we have
	$$ (\G_{n_2}-\G_{n_1})(\xi,\eta) \leq  \int_{\g_1(\xi,\eta)}(n_2-n_1)d|\gb| = \int_{\g_1(\xi,\eta)}\frac{(n_2-n_1)}{n_1}d|g_{n_1}|. $$
	This implies
	\begin{align*}
		(\G_{n_2}-\G_{n_1})^2(\xi,\eta) &\leq  \G_{n_1}(\xi,\eta)\int_{\g_1(\xi,\eta)}\frac{(n_2-n_1)^2}{n_1^2} d|g_{n_1}| \quad \textrm{(by Cauchy-Schwarz)}\\
		&=  \G_{n_1}(\xi,\eta)\int_0^{\G_{n_1}(\xi,\eta)}\frac{(n_2-n_1)^2}{n_1^2}(\g_1(\xi,\eta,t))dt \\
		&\leq \frac{\G_{n_1}(\xi,\eta)}{\l^2}\int_0^{\G_{n_1}(\xi,\eta)}(n_2-n_1)^2(\exp_{n_1}(\xi,tv_{n_1}(\xi,\eta)))dt,
	\end{align*}
	where $v_{n_1}(\xi,\eta) = \dot{\g}_{n_1}(\xi,\eta,0)$, that is, the unit tangent vector at $\xi$ that points towards $\eta$. This implies
	\begin{align}
		\int_{B_1(\xi)}(\G_{n_2}-\G_{n_1})^2(\xi,\eta)d\eta &\leq  \frac{\Lambda \diam_{\gb}(\O)}{\l^2} \int_{\p \O}\int_0^{\G_{n_1}(\xi,\eta)}(n_2-n_1)^2(\exp_{n_1}(\xi,tv_{n_1}(\xi,\eta)))dtd\eta \nonumber\\
		&= \frac{\L \diam_{\gb}(\O)}{\l^2} \int_{\p_{+}S_\xi\Oc}\int_0^{\tau_{n_1}(\xi,v)}(n_2-n_1)^2(\exp_{n_1}(\xi,tv))|\det[D_v\eta_{n_1}(\xi,v)]dt dv. \label{changeofvar}
	\end{align}
	by the change of variables formula. (Here, $d\eta$ is the surface measure on $\eta \in \p \O$ with respect to $\gb$.) We now find an upper bound for $|\det[D_v\eta_{n_1}]|$ on the support of the integrand. Recall that by definition,
	$$ \eta_{n_1}(\xi,v) = \exp_{n_1}(\xi, \tau_{n_1}(\xi,v)v).$$
	With the canonical identification of $T_v S_\xi\Oc$ with a subspace of $T_\xi \Oc$, we get
	\begin{align*}
        D_v\eta_{n_1}(\xi,v) &= D_w\exp_{n_1}(\xi,w)\big|_{w =\tau_{n_1}(\xi,v)v}\circ D_v(\tau_{n_1}(\xi,v)v) \\
		&= D_w \exp_{n_1}(\xi,w)\big|_{w=\tau_{n_1}(\xi,v)v} \circ \big( \tau_{n_1}(\xi,v)\id +v\otimes D_v\tau_{n_1}(\xi,v)\big).
	\end{align*}
	Here, $v\otimes D_v\tau_{n_1}(\xi,v)$ should be interpreted as the map
	$$w \in T_vS_\xi \Oc \sub T_\xi\Oc \qquad \mapsto \qquad [D_v\tau_{n_1}|_{(\xi,v)}(w)]v \in T_\xi \Oc. $$
	So we have
	\begin{align*}
		\|D_v\eta_{n_1}(\xi,v)\|_{op} &\leq  \left\|D_w\exp_{n_1}(\xi,w)\big|_{w=\tau_{n_1}(\xi,v)v}\right\|_{op}\big( \tau_{n_1}(\xi,v) +\|D_v\tau_{n_1}(\xi,v)\|_{op}\big) \\
		&\leq  L\left(\L\diam_{\gb}(\O) +\frac{L\L\diam_{\gb}(\O)}{\left\langle \nu(\eta_{n_1}(\xi,v)),u_{n_1}(\xi,v)\right\rangle_{\gb}}\right)
	\end{align*}
	by Lemma \ref{derivative of tau}. Now since $\O_0$ is a relatively compact subset of $\O$, there exists an $\varepsilon \in (0,1)$ such that if $\< \nu(\eta_{n_1}(\xi,v)),u_{n_1}(\xi,v)\>_{\gb} <\ve$, the geodesic $\g_{n_1}(\xi,v,\cdot)$ lies entirely within $\Oc \setminus \O_0$, and therefore,
	$$(n_2-n_1)^2(\exp_{n_1}(\xi,tv)) = 0 \qquad \textrm{for all } t \in [0,\tau_{n_1}(\xi,v)]. $$
	Therefore, on the support of the integrand in the right hand side of \eqref{changeofvar}, we have the bounds
	$$
	\|D_v\eta_{n_1}(\xi,v)\|_{op} \leq  L\left( \L\diam_{\gb}(\O) +\frac{L\L\diam_{\gb}(\O)}{\ve}\right) \lsim_{\O,\O_0,\gb,L} \L, $$
	and consequently
	$$|\det[D_v(\eta_{n_1}(\xi,v))]| \lsim_{\O,\O_0,\gb,L} \L^{m-1}.$$
	Applying this bound to the right hand side of \eqref{changeofvar}, we get
	\begin{align*}
		\int_{B_1(\xi)}(\G_{n_1}-\G_{n_2})^2(\xi,\eta)d\eta &\lsim  \frac{\L^m}{\l^2}\int_{\p_{+}S_{\xi}\Oc}\int_0^{\tau_{n_1}(\xi,v)}(n_2-n_1)^2(\exp_{n_1}(\xi,tv))dtdv \\
		&\sim  \frac{\L^m}{\l^2}\int_{\dom(\exp_{n_1}(\xi, \cdot))} \frac{(n_2-n_1)^2(\exp_{n_1}(\xi,w))}{|w|_{\gb}^{m-1}}dw
	\end{align*}
    Again by Remark \ref{d}, we have $(n_2-n_1)^2(\exp_{n_1}(\xi,w))=0$ for all $w \in \dom(\exp_{n_1}(\xi, \cdot))$ with $|w|_{\gb}\leq \delta$. Therefore, we get
	\[
	\int_{B_1(\xi)} (\G_{n_1}-\G_{n_2})^2(\xi,\eta)d\eta \lsim \frac{\L^m }{\l^2 \delta^{m-1}}\int_{\dom(\exp_{n_1}(\xi, \cdot))}(n_2-n_1)^2(\exp_{n_1}(\xi,w)) dw. \]
	We now make the change of variable $x = \exp_{n_1}(\xi, w)$. The assumption that $D_w \exp_{n_1}(\xi,w) \succ \ell$ implies that the inverse $w_{n_1}(\xi,\cdot)$ of $\exp_{n_1}(\xi,\cdot)$ satisfies $\|D_x w_{n_1}(\xi,x)\|_{op} < \ell^{-1}$, and consequently,
	\[ |\det(D_x w_{n_1}(\xi,x))| < \ell^{-m}. \]
	Therefore,
	\begin{align*}
		\int_{B_1(\xi)} (\G_{n_1}-\G_{n_2})^2(\xi,\eta)d\eta &\lsim  \frac{\L^m}{\l^2} \int_{\O}(n_2-n_1)^2(x)|\det(D_x w_{n_1}(\xi,x))|d\vol_{\gb}(x) \\
		& \lsim \frac{\L^m}{\l^2\ell^m}\int_{\O}(n_2-n_1)^2(x)d\vol_{\gb}(x).
	\end{align*}
	By analogous arguments, we also have
	\[
	\int_{B_2(\xi)}(\G_{n_1}-\G_{n_2})^2(\xi,\eta)d\eta \lsim \frac{\L^m}{\l^2 \ell^m}\int_{\O} (n_2-n_1)^2(x)d\vol_{\gb}(x). \]
	Adding the last two inequalities, we get
	\begin{align*}
		\int_{\p \O} (\G_{n_1}-\G_{n_2})^2(\xi,\eta) d\eta &\lsim  \frac{\L^m}{\l^2 \ell^m}\|n_1-n_2\|^2_{L^2(\O)} \\
		\Rightarrow \int_{\p \O}\int_{\p \O} (\G_{n_1}-\G_{n_2})^2(\xi,\eta)d\eta d\xi &\lsim  \frac{\L^m}{\l^2 \ell^m}\|n_1-n_2\|^2_{L^2(\O)} \\
		\Rightarrow \|\G_{n_1}-\G_{n_2}\|_{L^2(\p \O \times \p \O)} &\lsim_{\O,\O_0,\gb,\ell,L}  \frac{\L^{m/2}}{\l}\|n_1-n_2\|_{L^2(\O)}.
	\end{align*}
	This completes the proof.
	
\end{proof}

Next, we derive the analogous continuity estimate for the map $c \mapsto Z_c$. The key step is to show that for any $M>0$, the operator norm of the derivative of $\exp_{n_c}(x,v)$ is uniformly bounded for all $c \in \CM$ and $(x,v) \in \dom(\exp_{n_c})$. We begin with a simple lemma.

\begin{Lemma}\label{curvature Jacobi}
	Let $(\M,g)$ be a Riemannian manifold whose curvature tensor $R$ satisfies
	\[
	\|R\| = \sup \left\{ |R(u,v)w|_g : u,v,w \in S\M\right\} < \infty. \]
	Then any Jacobi field $J$ along a unit speed geodesic $\g:[0,T]\to \M$ satisfies the norm bounds
	\[
	|J(t)|_g^2+|\dot{J}(t)|_g^2 \leq e^{(1+\|R\|)t}\left(|J(0)|_g^2+|\dot{J}(0)|_g^2\right) \qquad \textrm{for all }t \in [0,T]. \]
\end{Lemma}
\begin{proof}
	Set $f(t) = |J(t)|_g^2+|\dot{J}(t)|_g^2$. Since $J$ is a Jacobi field, it satisfies the equation
	\[ \ddot{J}(t)+R(J(t),\dot{\g}(t))\dot{\g}(t) = 0. \]
	Therefore,
	\begin{align*}
		f'(t) &= 2\<J(t),\dot{J}(t)\>_g+2\<\dot{J}(t),\ddot{J}(t)\>_g \\
		&= 2\<J,\dot{J}\>_g+2\<\dot{J}, -R(J,\dot{\g})\dot{\g}\>_g \\
		&\leq  2|J|_g|\dot{J}|_g +2|\dot{J}|_g\|R\||J|_g|\dot{\g}|_g^2 \\
		&\leq  (1+\|R\|)f(t).
	\end{align*}
	So it follows that
	\[ f(t) \leq e^{(1+\|\R\|)t}f(0) \qquad \textrm{for all }t \in [0,T]. \]
\end{proof}

Next, let us recall the definition of the canonical metric on the tangent bundle of a Riemannian manifold, also called the Sasaki metric. Let $(\M,g)$ be a Riemannian manifold, $(x,w)\in T\M$, and $V_1,V_2 \in T_{(x,w)}T\M$. Then we may choose curves $\alpha_j(s)=(\sigma_j(s),v_j(s))$ in $T\M$, defined on $(-\ve,\ve)$, such that
$$\alpha_j(0) = (x,w), \qquad \dot{\alpha}_j(0)= V_j, \qquad \textrm{for }j=1,2.$$
The inner product of $V_1,V_2$ with respect to the Sasaki metric is defined to be
$$\< V_1, V_2\>_{g} := \<\dot{\sigma}_1(0),\dot{\sigma}_2(0)\>_{g}+ \<\dot{v}_1(0),\dot{v}_2(0)\>_{g}, $$
where $\dot{v}_j(s)$ represents the covariant derivative of $v_j(s)$ along the curve $\sigma_j(s)$. Note that we are using the same notation for the Sasaki metric as for the original metric $g$. Now, for any $C^1$ map $F:T\M \to \M$, the operator norm of the total derivative of $F$ at $(x,w)\in T\M$ is given by
$$\|DF(x,w)\|_{op} := \sup \{ |DF(x,w)(V)|_{g} \ : \ V \in T_{(x,w)}T\M, \, |V|_{g}=1\}.$$

We will show that if $c\in \CM$, the total derivative of $\exp_{n_c}$ is bounded above in the operator norm. 
\begin{Proposition}\label{ncbounded}
	For any $M>0$, there exists $L=L(M)>0$ such that for all $c\in \CM$, the total derivative of the exponential map of $g_{n_c}$ satisfies
	$$\|D \exp_{n_c}(x,w)\|_{op} <L $$
	for all $x \in \overline\Omega$ and $w \in \dom(\exp_{n_c}(x,\cdot))$. In particular, $n_c \in \N_{\lambda,\Lambda,\ell,L}(\O_0)$.
\end{Proposition}
\begin{proof}
	Suppose $c \in \CM$. Fix $(x,w)\in \dom(\exp_{n_c})$, and let $V \in  T_{(x,w)}T\Oc$. It suffices to show that
	\[ |D\exp_{n_c}(x,w)(V)|_{\gb} < L|V|_{\gb}. \] 
	Choose a curve $\alpha(s) = (\sigma(s),v(s))$ in $T\Oc$, defined on $(-\ve,\ve)$, such that $\alpha(0)=(x,w)$ and $\dot{\alpha}(0)=V$. Consider the family of geodesics $\Phi:(-\ve,\ve)\times[0,1] \to \Oc$ defined by 
	\[ \Phi(s,t) = \exp_{n_c}(\sigma(s),tv(s)). \]
	The variation field of this family of geodesics is
	\[ J(t) := \p_s\exp_{n_c}(\sigma(s),tv(s))\big|_{s=0}, \]
	which is a Jacobi field along $\g(t) := \Phi(0,t)$. Observe that
	\[ J(1) = \p_s \exp_{n_c}(\sigma(s), v(s))\big|_{s=0} = D\exp_{n_c}(x,w)(V), \]
	which is precisely the quantity whose norm we want to estimate. 
	
	Let $R$ be the Riemann curvature tensor of $(\Oc,g_{n_c})$, and let $R^i_{jkl}$ denote its tensor coefficients with respect to a fixed global coordinate chart on $\Oc$. Then we have
	\[ R^i_{jkl} = \p_k\G^i_{lj}-\p_l\G^i_{kj}+\G^i_{km}\G^m_{lj}-\G^i_{lm}\G^m_{kj}, \]
	where
	\[ \G^l_{jk} = \frac{1}{2}n_c^{-2}\gb^{lm}\left(\p_j(n_c^2\gb_{km})+\p_k(n_c^2\gb_{jm})-\p_m(n_c^2\gb_{jk})\right). \]
	This implies that for any $x\in \Oc$,
	\[
	\max_{ijkl}|R^i_{jkl}(x)| \lsim_{\gb} 1+ n_c(x)^{-2}\|n_c\|^2_{C^2} \lsim e^{4M}(1+M)^4. \]
	Therefore, for any $x\in \Oc$ and unit tangent vectors $u,v,w \in S_x\O$, 
	\begin{align*}
		|R(u,v)w|_{g_c} &\lsim  n_c(x)\left(\max_{ijkl}|R^i_{jkl}(x)u^jv^kw^l|\right) \lsim e^{5M}(1+M)^4 \\
		\Rightarrow \|R\| &\leq  Ce^{5M}(1+M)^4
	\end{align*}
	for some $C>0$. Taking $L^2 > \exp(1+C'e^{5M}(1+M)^4)$ and applying Lemma \ref{curvature Jacobi}, we get
	\begin{align*}
		|D\exp_c(x,w)(V)|_{g_c}^2 = |J(1)|_{g_{n_c}}^2 & <  L^2\left(|J(0)|_{g_{n_c}}^2+|\dot{J}(0)|_{g_{n_c}}^2\right)\\
		&= L^2\left( |\dot{\sigma}(0)|^2 +\left|\dot{v}(0)\right|^2\right) = L^2|V|_{\gb}^2.
	\end{align*}
	This completes the proof. 
\end{proof}

\begin{Corollary}\label{forward-c}
	There exists a constant $C_2' = C_2'(\O,\O_0,\gb,\ell,M)>0$ such that for all $c_1,c_2 \in \CM$,
	\[
	\|Z_{c_1}-Z_{c_2}\|_{L^2(\p \O \times \p \O)} \leq C_2'\|c_1-c_2\|_{L^2(\O_0)}. \]
\end{Corollary}
\begin{proof}
	We know from Theorem \ref{forward}, Proposition \ref{ncbounded}, and equation \eqref{link-L2} that
	\[
	\|\G_{c_1}-\G_{c_2}\|_{L^2(\p \Omega \times \p \Omega)} \lsim_{\O,\O_0,\gb,\ell,M}\|c_1-c_2\|_{L^2(\O_0)}. \]
	Now consider
	\[
	\|\G_{c_1}-\G_{c_2}\|_{L^2(\p \Omega \times \p \Omega)}^2 = \int_{\p \Omega \times \p \Omega}\left|e^{Z_{c_1}}-e^{Z_{c_2}}\right|^2\ d\xi d\eta. \]
	Recall that there exists $\delta>0$ such that $Z_{c_1}(\xi,\eta)=Z_{c_2}(\xi,\eta)$ whenever $\dist_{\gb}(\xi,\eta)<\delta$. On the set $\{\dist_{\gb}(\xi,\eta)\geq \delta\}$,
	\begin{align}
		e^{-M}\delta &\leq  \G_{c_j}(\xi,\eta) \leq e^M\diam_{\gb}(\O) \nonumber \\
		\Rightarrow -M+\log\delta &\leq  Z_{c_j}(\xi,\eta) \leq M+ \log|\diam_{\gb}(\O)|. \label{zcbound}
	\end{align}
   So by \eqref{lipexp},
    $$|e^{Z_{c_1}(\xi,\eta)}-e^{Z_{c_2}(\xi,\eta)}| \geq e^{-M}\delta|Z_{c_1}(\xi,\eta)- Z_{c_2}(\xi,\eta)|$$
	for all $(\xi,\eta) \in \p\O \times \p\O$. Consequently,
	\[
	\|\G_{c_1}-\G_{c_2}\|^2_{L^2(\p \Omega \times \p \Omega)} = \int \left|e^{Z_{c_1}}-e^{Z_{c_2}}\right|^2\ d\xi d\eta \geq e^{-2M}\delta^2\int |Z_{c_1}-Z_{c_2}|^2\ d\xi d\eta. \]
   So we conclude that
	\[
	\|Z_{c_1}-Z_{c_2}\|_{L^2} \lsim \|\G_{c_1}-\G_{c_2}\|_{L^2} \lsim \|c_1-c_2\|_{L^2}. \]
\end{proof}

We conclude this section with a technical result that will be necessary for the proof of Theorem \ref{postcontr} in Section 3.

\begin{Theorem}\label{h2bounds}
	Given $M>0$, there exists a constant $C_3' = C_3'(\O,\O_0, \gb,\ell,M)>0$ such that for all $c_1,c_2 \in \CM$,
	\[
	\|Z_{c_1}-Z_{c_2}\|_{H^2(\p \O \times \p \O)} \leq C_3'. \]
\end{Theorem}

\begin{proof}
	We know from Theorem \ref{forward} that 
	\[ \|Z_{c_1}-Z_{c_2}\|_{L^2} \lsim \|c_1-c_2\|_{L^2} \lsim 2M. \]

    Next, let $\xi,\eta \in \p \O$. It follows from Remark \ref{d} that if $\dist_{\gb}(\xi,\eta)<\delta$, then $Z_{c_1}-Z_{c_2}$ and all its derivatives are identically $0$ in a neighborhood of $(\xi,\eta)$. On the other hand, if $\dist_{\gb}(\xi,\eta)>\delta$, Lemma \ref{derivative estimates} implies
	\[
	|d_\xi(Z_{c_1}-Z_{c_2})(\xi,\eta)|_{\gb} \leq \frac{|d_\xi \G_{c_1}(\xi,\eta)|_{\gb}}{\G_{c_1}(\xi,\eta)} +\frac{|d_\xi \G_{c_2}(\xi,\eta)|_{\gb}}{\G_{c_2}(\xi,\eta)} \lsim \frac{e^M}{\delta}. \]
	This shows that $\|d_\xi(Z_{c_1}-Z_{c_2})\|_{L^2}$ is uniformly bounded for $c_1,c_2 \in \CM$. By symmetry, $\|d_\eta(Z_{c_1}-Z_{c_2})\|_{L^2}$ is also uniformly bounded.
 
      So it only remains to consider the Hessian tensor of $Z_{c_1}-Z_{c_2}$. Let $\nabla$ denote the Levi-Civita connection on $\p \O_{\xi}\times \p \O_{\eta}$, and let $\pi^\xi:\p \O_{\xi}\times \p \O_{\eta} \to \p \O_{\xi}$ and $\pi^\eta:\p \O_{\xi}\times \p \O_{\eta} \to \p \O_{\eta}$ denote the canonical projection maps. We may decompose $\nabla$ as $\nabla^\xi +\nabla^{\eta}$, where $\nabla^{\xi}$ and $\nabla^{\eta}$ are the covariant derivative operations with respect to $\xi$ and $\eta$ respectively. More precisely, given any tensor field $F$ on $\p \O_\xi \times \p \O_\eta$, and any tangent vector $v\in T(\p \O_\xi \times \p \O_\eta)$, we have
      $$\nabla^{\xi}_vF = \nabla_{(\pi^\xi)_*v_\xi}F, \qquad \nabla^{\eta}_vF = \nabla_{(\pi^\eta)_*v_\eta}F,$$
      where $(v_\xi,v_\eta)$ is the image of $v$ under the canonical isomorphism from $T(\p\O_\xi\times\p\O_\eta)$ to $(T\p\O_\xi)\times (T\p\O_\eta)$. Correspondingly, the Hessian operator on $\p \O_\xi \times \p \O_\eta$ can be decomposed as
      \begin{align*}
          \hess = \nabla^2 &= (\nabla^\xi+\nabla^\eta)(\nabla^\xi+\nabla^\eta) \\
          &= \nabla^\xi \nabla^\xi +\nabla^\xi\nabla^\eta + \nabla^\eta\nabla^\xi +\nabla^\eta\nabla^\eta \\
          &= \hess_\xi +\nabla^\xi\nabla^\eta+ \nabla^\eta\nabla^\xi +\hess_\eta,
      \end{align*}
      where $\hess_\xi$ and $\hess_\eta$ are the Hessian operators with respect to $\xi$ and $\eta$ respectively. Now let $\xi,\eta \in \p \O$ be such that $\dist_{\gb}(\xi,\eta)>\delta$. Then for $j=1,2$,
      \begin{align*}
          \nabla^\xi\nabla^\eta Z_{c_j}(\xi,\eta) &= \nabla^\xi\nabla^\eta \log\G_{c_j}(\xi,\eta) \\
          &= \left(\frac{\nabla^\xi\nabla^\eta\G_{c_j}}{\G_{c_j}} - \frac{d_\xi\G_{c_j} \otimes d_\eta \G_{c_j}}{\G_{c_j}^2}\right)(\xi,\eta).
      \end{align*}
      By Lemma \ref{derivative estimates}, this implies
      \begin{align*}
          |\nabla^\xi\nabla^\eta Z_{c_j}(\xi,\eta)|_{\gb} &\leq \frac{|\nabla^\xi\nabla^\eta\G_{c_j}(\xi,\eta)|_{\gb}}{\G_{c_j}(\xi,\eta)} + \frac{|d_\xi \G_{c_j}(\xi,\eta)|_{\gb}|d_\eta \G_{c_j}(\xi,\eta)|_{\gb}}{\G_{c_j}^2(\xi,\eta)} \\
          &\lsim \frac{1+\ell^{-1}}{\l\delta^2} + \frac{1}{\delta^2}.
      \end{align*}
    This implies that $\|\nabla^\xi\nabla^\eta (Z_{c_1}-Z_{c_2})\|_{L^2}$ is uniformly bounded as well. Finally, consider the fact \cite{V11} that
	\[
	\hess_\xi \G_{c_j}(\xi,\eta) = (D_w\exp_{c_j}(\xi,w(\xi,\eta)))^{-1}(D_\xi \exp_{c_j}(\xi, w(\xi,\eta))), \]
	where $w(\xi,\cdot)$ is the inverse of $\exp_{c_j}(\xi,\cdot)$ as in Lemma \ref{derivative estimates}. Therefore, by Proposition \ref{ncbounded},
    $$|\hess_\xi \G_{c_j}(\xi,\eta)|_{\gb} \lsim \ell^{-1}L(M).$$
    Writing $Z_{c_j}=\log \G_{c_j}$, we get
    \begin{align*}
        \hess_\xi Z_{c_j}(\xi,\eta) &= \hess_{\xi} \log \G_{c_j}(\xi,\eta) \\
        &= \left(\frac{\hess_\xi \G_{c_j}}{\G_{c_j}} - \frac{d_\xi \G_{c_j}\otimes d_\xi \G_{c_j}}{\G_{c_j}^2}\right)(\xi,\eta),
    \end{align*}
    which implies
    \begin{align*}
        |\hess_\xi Z_{c_j}(\xi,\eta)|_{\gb} &\leq \frac{|\hess_\xi \G_{c_j}(\xi,\eta)|_{\gb}}{\G_{c_j}(\xi,\eta)} + \frac{|d_\xi\G_{c_j}(\xi,\eta)|_{\gb}^2}{\G_{c_j}^2(\xi,eta)} \\
        &\lsim \frac{\ell^{-1}L}{\lambda \delta^2} +\frac{1}{\delta^2}.
    \end{align*}
	So we conclude that $\|\hess_\xi(Z_{c_1}-Z_{c_2})\|_{L^2}$, and by similar arguments, $\|\hess_\eta(Z_{c_1}-Z_{c_2})\|_{L^2}$, are both uniformly bounded on $\CMb$ as well. This proves the result.
\end{proof}

\section{Statistical Inversion through the Bayesian framework}\label{stats}

As discussed in the Introduction, we will be using the posterior mean of $c$ given finitely many measurements $\mathcal{D}_N = (X_i,Y_i,Z_i)_{i=1}^N$, as an estimator for the true metric parameter $c_0$. Let us begin by describing the prior distribution $\Pi$ for $c \in C^3_0(\O_0)$. We will assume that $\Pi$ arises from a centered Gaussian probability distribution $\widetilde{\Pi}$ on the Banach space $C(\overline{\O}_0)$ that satisfies the following conditions.

\begin{cond}
    Let $\beta\geq 3$ and $\alpha > \beta +\frac{m}{2}$. We assume that $\widetilde{\Pi}$ is a centered Gaussian Borel probability measure on $C(\overline{\O}_0)$ that is supported in a separable subspace of $C^\beta_0(\O_0)$. Moreover, its {\it Reproducing Kernel Hilbert space (RKHS)} $(\mathcal{H},\|\cdot\|_{\mathcal{H}})$ must be continuously embedded in the Sobolev space $H^\alpha(\O_0)$. 
\end{cond}

 We refer the reader to \cite[Chapter 11]{GvdV17} or \cite[Sections 2.1 and 2.6]{GN15} for basic facts about Gaussian probability measures and their Reproducing Kernel Hilbert Spaces.

We now define the prior $\Pi$ to be the restriction of $\widetilde{\Pi}$ to $\CMb$ in the sense that
\begin{equation}\label{Pi}
   \Pi(A) = \frac{\widetilde{\Pi}\left(A\cap \CMb\right)}{\widetilde{\Pi}(\CMb)} 
\end{equation}
for all Borel sets $A \sub C^3_0(\O_0)$. We will see in Lemma \ref{smallball} that $C^\beta$-balls have positive $\widetilde{\Pi}$-measure. This together with the fact that $\CMb$ is an open subset of $C_0^\beta(\O_0)$ (c.f.\ Remark \ref{open}) implies that $\widetilde{\Pi}(\CMb) >0$. Therefore, \eqref{Pi} yields a well-defined probability distribution on $C^3_0(\O_0)$.

\begin{Theorem}\label{main}
	Let $\Pi$ be a prior distribution on $C^3_0(\O_0)$ defined by \eqref{Pi}. Assume that the true parameter $c_0 \in \CMb \cap \H$, and let $\overline{c}_N$ be the mean \eqref{postmean} of the posterior distribution $\Pi(\cdot |\D_N)$ arising from observations \eqref{obs}. Then there exists $\omega\in(0,1/4)$ such that
	\[
	P^N_{c_0}\left( \|\overline{c}_N-c_0\|_{L^2(\O_0)} > N^{-\omega}\right) \to 0 \qquad \textrm{as } N \to \infty. \]
	Moreover, $\omega$ can be made arbitrarily close to $1/4$ for $\alpha$, $\beta$ large enough.
\end{Theorem}

\begin{remark}
    The assumption that $c_0\in \CMb \cap\mathcal{H}$ is weaker than in Theorem \ref{mainth}, where we assumed that $c_0$ is smooth, compactly supported in $\O_0$, and that $g_{n_{c_0}}$ is simple. Indeed, if $g_{n_{c_0}}$ is a smooth simple metric, $c_0$ necessarily belongs to $\CMb$ for appropriate values of $\ell,M$, and any $\beta$. Moreover, given any $c_0\in H^\alpha_0(\O_0)$, it is possible to choose $\widetilde{\Pi}$ so that its RKHS $\mathcal{H}$ contains $c_0$. Indeed, let $(f(x): x \in \O_0)$ be the so-called \emph{Mat\'{e}rn-Whittle process of regularity $\alpha$} (see \cite[Example 11.8]{GvdV17}), whose corresponding RKHS is $H^\alpha(\O_0)$. It follows from Lemma I.4 in \cite{GvdV17} that the sample paths of this process belong almost surely to $C^\beta(\overline{\O}_0)$. Now choose a cut-off function $\varphi \in C^\infty(\overline{\O}_0)$ such that $\varphi >0$ on $\O_0$, $\varphi$ and all its partial derivatives vanish on $\p\O_0$, and $\varphi^{-1}c_0 \in H^\alpha(\O_0)$. Define $\widetilde{\Pi}$ to be the probability law of $(\varphi(x)f(x): x \in \O_0)$. Then $\H = \left\{\varphi f  :  f \in H^\alpha(\O_0) \right\}$, which contains $c_0$. Therefore, Theorem \ref{main} is a more general and precise version of Theorem \ref{mainth}.
\end{remark}

\subsection{A General Contraction Theorem} \label{gencontr}

Our proof of Theorem \ref{main} will follow the same general strategy as in \cite{MNP21a}, with some modifications necessitated by the fact that our prior $\Pi$ is not in itself a Gaussian probability measure, but rather the restriction of such a measure to $\CMb$. We begin with a general posterior contraction result (Theorem \ref{g-contr}). This is a simplified version of \cite[Theorem 5.13]{MNP21a}, which suffices for us since our prior $\Pi$ independent of $N$. Before stating the result, we need to introduce some notation. Recall that for $c \in \CMb$, we defined $p_c$ as the probability density function 
$$p_c(x,y,z) = \frac{1}{\sqrt{2\pi}}\exp\left\{-\frac{1}{2}(z-Z_c(x,y))^2\right\} \qquad \textrm{for all } (x,y,z)\in \X, $$
where $\X = \p\O \times \p\O \times \R$. Given $c_1,c_2 \in \CMb$, let
$$ h(c_1,c_2) := \left(\int_{\X}(\sqrt{p_{c_1}}-\sqrt{p_{c_2}})^2 d\mu(x,y)\, dz\right)^{1/2} $$
denote the Hellinger distance between $p_{c_1}$ and $p_{c_2}$, 
$$ K(c_1,c_2) := \E_{c_1}\left[\log\left(\frac{p_{c_1}}{p_{c_2}}\right)\right] = \int_{\X}\log\left(\frac{p_{c_1}}{p_{c_2}}\right)p_{c_1}d\mu(x,y)\, dz$$
the Kullback-Leibler divergence, and
$$V(c_1,c_2) := \E_{c_1}\left[\log\left(\frac{p_{c_1}}{p_{c_2}}\right)\right]^2. $$
Also, for any $F \subseteq \CMb$ and $\delta >0$, we let $\N(F, h, \delta)$ denote the minimum number of $h$-balls of radius $\delta$ needed to cover $F$. 

\begin{Theorem}\label{g-contr}
	Let $\widehat{\Pi}$ be a Borel probability measure on $C_0^3(\O_0)$ supported on $\CMb$. Let $c_0 \in \CMb$ be fixed, and let $\delta_N$ be a sequence of positive numbers such that $\delta_N \to 0$ and $\sqrt{N}\delta_N \to \infty$ as $N \to \infty$. Assume that the following two conditions hold:
	\begin{enumerate}[(1)]
		\item There exists $C>0$ such that for all $N \in \mathbb{N}$,
		\begin{equation}\label{pmass}
			\widehat{\Pi}\left(\left\{c \in \CMb: K(c,c_0) \leq \delta_N^2, V(c,c_0) \leq \delta_N^2 \right\}\right) \geq e^{-CN\delta_N^2}.
		\end{equation}
		\item There exists $\widetilde{C}>0$ such that
		\begin{equation}\label{pcomp}
			\log \N(\CMb, h, \delta_N) \leq \widetilde{C}N\delta_N^2.
		\end{equation}
  \end{enumerate}
		Now suppose that we make i.i.d. observations $\mathcal{D}_N = (X_i,Y_i,Z_i)_{i=1}^N \sim P^N_{c_0}$. Then for some $k>0$ large enough, we have
		\begin{equation}
			P^N_{c_0}\left( \widehat{\Pi}\left( \left\{c \in \CMb : h(c,c_0) \leq k\delta_N\right\}|\mathcal{D}_N\right) \leq 1-e^{-(C+3)N\delta_N^2}\right) \to 0
		\end{equation}
		as $N \to \infty$.
\end{Theorem}
\begin{proof}
	Define
	\begin{equation}\label{bn}
		B_N = \left\{c \in \CMb :  K(c,c_0) \leq \delta_N^2, V(c,c_0) \leq \delta_N^2 \right\}, \qquad N \in \mathbb{N}. 
	\end{equation}
	By condition (1) and \cite[Lemma 7.3.2]{GN15}, we have that for any $\zeta >0$ and any probability measure $\widetilde{m}$ on $B_N$,
	$$ P^N_{c_0}\left( \int_{B_N} \prod_{i=1}^N \frac{p_c}{p_{c_0}}(X_i,Y_i,Z_i)d\widetilde{m}(c) \leq e^{-(1+\zeta)N\delta_N^2}\right) \leq \frac{1}{\zeta^2 N\delta_N^2}. $$
	In particular, choosing $\zeta =1$ and taking $\widetilde{m}$ to be the restriction of $\widehat{\Pi}$ to $B_N$ followed by normalization, we get that
	$$ P^N_{c_0}\left( \int_{B_N} \prod_{i=1}^N \frac{p_c}{p_{c_0}}(X_i,Y_i,Z_i)d\widehat{\Pi}(c) \leq \widehat{\Pi}(B_N) e^{-2N\delta_N^2}\right) \leq \frac{1}{N\delta_N^2} \xrightarrow{N \to \infty} 0. $$
	Set
	$$A_N = \left\{\int_{B_N} \prod_{i=1}^N \frac{p_c}{p_{c_0}}(X_i,Y_i,Z_i)d\widehat{\Pi}(c) \geq e^{-(2+C)N\delta_N^2}\right\},$$
	where $C$ is as in condition (1). It is clear that $A_N \supseteq \left\{ \int_{B_N} \prod_{i=1}^N\frac{p_c}{p_{c_0}}d\widehat{\Pi}(c) \geq \widehat{\Pi}(B_N)e^{-2N\delta_N^2}\right\}$, and therefore, $P^N_{c_0}(A_N) \to 1$ as $N \to \infty$.
	
	Next, we consider condition (2). Let $k > k' > 0$ be numbers to be determined later. Fix $N$ and define the function $N(\ve) = e^{\widetilde{C}N\delta_N^2}$ for all $ \ve > \ve_0 = k'\delta_N$. It follows from condition (2) that for any $\ve > \ve_0$,
	$$\N(\CMb, h, \ve/4) \leq \N(\CMb, h, k'\delta_N/4) \leq e^{\widetilde{C}N\delta_N^2} = N(\ve). $$
	Therefore, by \cite[Theorem 7.1.4]{GN15}, there exist test functions $\Psi_N = \Psi_N(\mathcal{D}_N)$ such that for some $K >0$,
	\[
	P^N_{c_0}[\Psi_N = 1] \leq \frac{N(\ve)}{K}e^{-KN\ve^2} \quad ; \quad \sup_{c: h(c,c_0) > \ve}\E^N_c[1-\Psi_N] \leq e^{-KN\ve^2}.
	\]
	Now let $l > \widetilde{C}$ be arbitrary. Setting $k = \sqrt{l / K}$ and $\ve = k\delta_N$, we can see that this implies
	\begin{equation}\label{tests}
		P^N_{c_0}[\Psi_N = 1] \to 0 \, \textrm{ as } N \to \infty \quad ; \quad \sup_{c: h(c,c_0)>k\delta_N}\E^N_c[1-\Psi_N] \leq e^{-l N\delta_N^2}. 
	\end{equation}
	Now define
	\[
	F_N = \{ c \in \CMb: h(c,c_0) \leq k\delta_N\} \]
	which is the event whose probability we want to bound. Then by \eqref{tests},
	\begin{align*}
		& \  P^N_{c_0}\left( \widehat{\Pi}(F_N^c|\D_N) \geq e^{-(C+3)N\delta_N^2}\right) \\
		= & \ P^N_{c_0}\left( \frac{\int_{F_N^c}\prod_{i=1}^N\frac{p_c}{p_{c_0}}(X_i,Y_i,Z_i)d\widehat{\Pi}(c)}{\int \prod_{i=1}^N \frac{p_c}{p_{c_0}}(X_i,Y_i,Z_i)d\widehat{\Pi}(c)} \geq e^{-(C+3)N\delta_N^2}, \ \Psi_N =0, \ A_N\right) +o(1) \\
		\leq & \ P^N_{c_0}\left((1-\Psi_N)\int_{F_N^c}\prod_{i=1}^N\frac{p_c}{p_{c_0}}(X_i,Y_i,Z_i)d\widehat{\Pi}(c) \geq e^{-(2C+5)N\delta_N^2}\right) +o(1).
	\end{align*}
	Now by Markov's inequality, this is further bounded above by
	\begin{align*}
		 & \ \E^N_{c_0}\left[(1-\Psi_N)\int_{F_N^c}\prod_{i=1}^N \frac{p_c}{p_{c_0}}(X_i,Y_i,Z_i)d\widehat{\Pi}(c)\right]e^{(2C+5)N\delta_N^2} +o(1) \\
		= & \ \left[\int_{F_N^c} \E^N_{c_0}\left[(1-\Psi_N)\prod_{i=1}^N\frac{p_c}{p_{c_0}}(X_i,Y_i,Z_i)\right]d\widehat{\Pi}(c)\right]e^{(2C+5)N\delta_N^2} +o(1) \quad \textrm{(by Fubini's Theorem)} \\
		= & \ \left[ \int_{c: h(c,c_0) > k\delta_N} \E^N_c[(1-\Psi_N)]d\widehat{\Pi}(c)\right]e^{(2C+5)N\delta_N^2} +o(1) \\
		\leq & \ e^{(2C+5-l)N\delta_N^2} +o(1).
	\end{align*}
	Now choosing $l > 2C+5$, the Theorem follows.
\end{proof}

\subsection{Properties of the Prior} \label{priorprop}

In this section, we will verify the assumptions of Theorem \ref{g-contr} when $\widehat{\Pi}=\Pi$. The key ingredient in the arguments is the forward continuity estimate from Corollary \ref{forward-c}. We begin by observing that the Hellinger distance between $c_1, c_2 \in \CMb$ is equivalent to the $L^2(\p \O \times \p \O)$ distance between $Z_{c_1}$ and $Z_{c_2}$.

\begin{Lemma}\label{hellinger-l2}
	There exists $\kappa=\kappa(\O,\gb,M)>0$ such that for all $c_1,c_2 \in \CMb$,
	$$\kappa\|Z_{c_1}-Z_{c_2}\|_{L^2}^2 \leq h^2(c_1,c_2) \leq \frac{1}{4\vol_{\gb}(\p\O)^2}\|Z_{c_1}-Z_{c_2}\|_{L^2}^2. $$
\end{Lemma}
\begin{proof}
	Consider the ``Hellinger affinity'' function
	$$\rho(c_1,c_2) = \int_{\X}\sqrt{p_{c_1}p_{c_2}}d\mu = 1-\frac{1}{2}h^2(c_1,c_2). $$
	We have
	\begin{align}
		\rho(c_1,c_2) =& \ \frac{1}{\sqrt{2\pi}}\int_{\X}\exp\left\{-\frac{1}{4}((z-Z_{c_1}(x,y))^2+(z-Z_{c_2}(x,y))^2)\right\} d\mu(x,y)\, dz \nonumber \\
		=& \ \frac{1}{\vol_{\gb}(\p\O \times \p\O)}\int_{\p \O \times \p \O} \exp\left\{-\frac{1}{4}(Z_{c_1}(x,y)^2+Z_{c_2}(x,y)^2)\right\} \nonumber\\
		   & \times \left[\frac{1}{\sqrt{2\pi}}\int_{-\infty}^{\infty}\exp\left\{-\frac{1}{2}\left(z-\frac{Z_{c_1}+Z_{c_2}}{2}\right)^2\right\}dz\right]\exp\left\{\frac{1}{8}(Z_{c_1}+Z_{c_2})^2\right\}dx\, dy \nonumber \\
		=& \ \frac{1}{\vol_{\gb}(\p\O)^2}\int_{\p \O \times \p \O}\exp\left\{-\frac{1}{8}(Z_{c_1}(x,y)-Z_{c_2}(x,y))^2\right\}dx\, dy \label{rhoform}.
	\end{align}
	Now applying the simple estimate $e^{-t} \geq 1-t$ for all $t \geq 0$, we get
	\begin{align*}
		\rho(c_1,c_2) &\geq \frac{1}{\vol_{\gb}(\p\O)^2}\int_{\p \O \times \p \O}\left[1-\frac{1}{8}(Z_{c_1}-Z_{c_2})^2\right] dx \, dy \\
		&= 1-\frac{1}{8\vol_{\gb}(\p\O)^2}\|Z_{c_1}-Z_{c_2}\|_{L^2}^2.
	\end{align*}
	Consequently,
	$$ h^2(c_1,c_2) = 2(1-\rho(c_1,c_2)) \leq \frac{1}{4\vol_{\gb}(\p\O)^2}\|Z_{c_1}-Z_{c_2}\|_{L^2}^2. $$
	Next, we use the fact $Z_{c_1},Z_{c_2}$ satisfy the uniform bounds \eqref{zcbound} on the support of $Z_{c_1}-Z_{c_2}$. Consequently, for all $x,y\in \p\O$, we have
    \begin{equation}
        |Z_{c_1}(x,y)-Z_{c_2}(x,y)|\leq \Delta, \label{zcdiff}
    \end{equation}
    where $\Delta = 2M+\log \diam_{\gb}(\O)-\log\delta$. Set $T=\Delta^2/8$ and observe that for all $t \in [0,T]$,
	$$e^{-t} \leq 1- \left(\frac{1-e^{-T}}{T}\right)t $$
	by the convexity of $t \mapsto e^{-t}$. Therefore, for $\kappa = \frac{1-e^{-T}}{4T}$,  we have
	\[
	\exp\left\{-\frac{1}{8}(Z_{c_1}(x,y)-Z_{c_2}(x,y))^2\right\} \leq 1-\frac{\kappa}{2}|Z_{c_1}(x,y)-Z_{c_2}(x,y)|^2 \]
	for all $(x,y) \in \p \O \times \p \O$. Integrating both sides of this inequality with respect to $d\mu(x,y)$ and applying \eqref{rhoform}, we get
	\begin{align*}
		\rho(c_1,c_2)  &\leq 1- \frac{\kappa}{2}\|Z_{c_1}-Z_{c_2}\|_{L^2}^2 \\
		\Rightarrow \ h^2(c_1,c_2) & \geq  \kappa\|Z_{c_1}-Z_{c_2}\|_{L^2}^2.
	\end{align*}
	This completes the proof.
\end{proof}

Now let us verify Condition (1) of Theorem \ref{g-contr} for $\Pi$.

\begin{Lemma}\label{smallball0}
	For $c_0 \in \CMb$ and $t>0$, define
	\begin{equation*}
		\mathcal{B}_N(t) = \{c \in \CMb : \|c-c_0\|_{C^\beta} \leq \delta_N/t \}, 
	\end{equation*}
	and let $B_N, \Pi,$ and $\delta_N$ be as in Theorem \ref{g-contr}. Then for some $t>0$ large enough, $\mathcal{B}_N(t) \subset B_N$ for all $N \in \mathbb{N}$. In particular,
	$$\Pi(B_N) \geq \Pi(\mathcal{B}_N(t)). $$
\end{Lemma}
\begin{proof}
	We need to verify that if $t$ is large enough, then for any $c\in \mathcal{B}_N(t)$, we have $K(c,c_0) \leq \delta_N^2$ and $V(c,c_0) \leq \delta_N^2$. Consider a random observation $(X,Y,Z)$, where $(X,Y)$ is a pair of boundary points chosen with respect to the uniform probability measure $\mu$, and $Z= Z_{c_0}(X,Y)+\epsilon$, with $\epsilon \sim N(0,1)$ independent of $(X,Y)$. Observe that for any $c \in \mathcal{B}_N(t)$,  
	\begin{align}
		\log \frac{p_{c_0}}{p_{c}}(X,Y,Z) &= -\frac{1}{2}[(Z-Z_{c_0}(X,Y))^2-(Z-Z_c(X,Y))^2] \nonumber\\
		&= \frac{1}{2}(Z_c(X,Y)-Z_{c_0}(X,Y))^2 -\e(Z_c(X,Y)-Z_{c_0}(X,Y)). \label{pcpform}
	\end{align}
    Since $\E[\e|X,Y]=0$, we have
	\begin{align}
        K(c,c_0) &= \E_{c_0}\left[\log\frac{p_{c_0}}{p_c}(X,Y,Z)\right] \nonumber \\
		  &= \E^{\mu}\left[\frac{1}{2}(Z_c(X,Y)-Z_{c_0}(X,Y))^2\right] \label{kform}\\
		&= \frac{1}{2\vol_{\gb}(\p\O\times \p\O)}\int_{\p \O \times \p \O}(Z_c(x,y)-Z_{c_0}(x,y))^2 \, dx \, dy \nonumber\\
		&= \frac{1}{2\vol_{\gb}(\p\O)^2}\|Z_c-Z_{c_0}\|_{L^2}^2 \nonumber\\
		&\lsim \|c-c_0\|^2_{L^2} \qquad (\textrm{by Corollary \ref{forward-c}}) \nonumber\\
        &\lsim \frac{\delta_N^2}{t^2}.\label{kdelta}
	\end{align}
	So it follows that if $t$ is large enough, $K(c,c_0)\leq \delta_N^2$ for all $c \in \mathcal{B}_N(t)$. Next, consider
	\begin{align*}
        V(c,c_0) &= \E_{c_0}\left[\log \frac{p_{c_0}}{p_{c}}(X,Y,Z)\right]^2 \\
        &\leq  2\E^\mu\left[\frac{1}{2}(Z_{c}-Z_{c_0})^2\right]^2 +2\E^\mu\left[(Z_c-Z_{c_0})^2\E_\e[\e^2]\right] \qquad \textrm{(by \eqref{pcpform})} \\
        &= \frac{1}{2}\int_{\p\O \times \p\O}|Z_c-Z_{c_0}|^4d\mu(x,y) + 2\E^{\mu}[Z_c-Z_{c_0}]^2 \qquad \textrm{(since $\E[\e^2]=1$)}\\
        &\leq \frac{\|Z_c-Z_{c_0}\|^2_{L^\infty}}{2\vol_{\gb}(\p\O)^2}\|Z_c-Z_{c_0}\|^2_{L^2}+4K(c,c_0)
    \end{align*}
    by \eqref{kform}. It follows from \eqref{zcdiff} that $\|Z_c-Z_{c_0}\|_{L^\infty}< \Delta$, where $\Delta>0$ depends only on $\O,\gb,\delta$. Consequently, 
	\begin{align*}
        V(c,c_0) &\lsim  \|Z_c-Z_{c_0}\|^2_{L^2} +K(c,c_0) \\
		&\lsim  C_2'^2\|c-c_0\|^2_{L^2} +K(c,c_0) \qquad (\textrm{by Corollary \ref{forward-c}}) \\
        &\lsim \|c-c_0\|^2_{C^\beta} + \frac{\delta_N^2}{t^2} \qquad \textrm{(by \eqref{kdelta})} \\
        &\lsim \frac{\delta_N^2}{t^2}.
	\end{align*}
	This shows that for $t>0$ large enough, we also get $V(c,c_0)\leq \delta_N^2$ for all $c \in \mathcal{B}_N(t)$.
\end{proof}
Next, we will establish a lower bound for $\Pi(\mathcal{B}_N(t))$, which will follow from estimates of $\widetilde{\Pi}$- measures of sets of the form $\{c  :  \|c\|_{C^\beta}\leq \ve\}$ when $\ve>0$ is small. To this end, it is convenient to work with H\"{o}lder-Zygmund spaces $C^s_*(\O_0)$, with $s>0$ (see \cite{Triebel} for a detailed treatment). If $s$ is not an integer, $C^s_*(\O_0)$ is simply the H\"{o}lder space $C^s(\overline{\O}_0)$. On the other hand, if $s$ is a positive integer, $C^s_*(\O_0)$ is a larger space than $C^s(\overline{\O}_0)$, and is defined by the norm
$$\|f\|_{C^s_*(\O_0)} = \sum_{|a|\leq s-1}\sup_{x \in \O_0}|\p^a f(x)| + \sum_{|a|= s-1}\sup_{x\in \O_0, \, h\neq 0}\frac{|\p^af(x+h)+\p^af(x-h)-2f\p^a(x)|}{|h|}.$$
In either case, it is easy to see that $\|f\|_{C^s_*}\leq \|f\|_{C^s}$ for all $f \in C^s(\overline{\O}_0)$. It turns out that $C^s_*(\O_0)$ coincides with the Besov space $B^s_{\infty,\infty}(\O_0)$, which allows us to use various embedding and approximation results from Besov space theory. 

Before proceeding, let us fix $\nu>0$ such that
\begin{equation}\label{nu}
	\nu > \max\left\{\frac{2m}{2(\alpha-\beta)-m},\frac{m}{\beta}\right\}, \qquad \textrm{and define} \qquad \delta_N = N^{-1/(2+\nu)}. 
\end{equation}
It is easy to verify that $\delta_N \to 0$ and $\sqrt{N}\delta_N = N^{\tfrac{\nu}{2(2+\nu)}} \to \infty$ as $N \to \infty$.

\begin{Lemma}\label{smallball}
	Let $c_0 \in \CMb\cap \H$, and define $\delta_N$ as in \eqref{nu}. Then for $t>0$ large enough, there exists $C' = C'(\O,\O_0, \gb,\alpha,\beta,\ell,M,c_0,t)>0$ such that for all $N \in \mathbb{N}$,
	\[
	\Pi(\mathcal{B}_N(t)) \geq \exp\{-C'N\delta_N^2\}. \]
	In particular, there exists $C = C(\O,\O_0,\gb, \alpha, \beta, \ell, M, c_0)>0$ such that for all $N \in \mathbb{N}$,
	\[
	\Pi(B_N) \geq \exp\{-CN\delta_N^2\}. \]
\end{Lemma}
\begin{proof}
	The sets $\{b \in C^3_0(\O_0): \|b\|_{C^\beta}\leq \delta\}$ for $\delta>0$ are convex and symmetric. Hence by \cite[Corollary 2.6.18]{GN15},
	$$ \widetilde{\Pi}(\|c-c_0\|_{C^\beta} \leq \delta_N/t) \geq e^{-\|c_0\|_{\H}^2/2}\widetilde{\Pi}(\|c\|_{C^\beta} \leq \delta_N/t). $$
	Moreover, since $c_0 \in \CMb$, which is open with respect to the $C^\beta$ metric, we have for all sufficiently large $t >0$,
	$$\Pi(\mathcal{B}_N(t)) = \Pi(\|c-c_0\|_{C^\beta}\leq \delta_N/t) = \frac{ \widetilde{\Pi}(\|c-c_0\|_{C^\beta}\leq \delta_N/t)}{\widetilde{\Pi}(\CMb)}, $$
	and therefore,
 \begin{equation}
     \Pi(\mathcal{B}_N(t)) \geq e^{-\|c_0\|_{\H}^2/2}\frac{\widetilde{\Pi}(\|c\|_{C^\beta}\leq \delta_N/t)}{\widetilde{\Pi}(\CMb)}. \label{pibn}
 \end{equation}
    Next, choose a real number $\gamma$ such that
    \begin{equation}
       \beta < \gamma < \alpha-\frac{m}{2}, \qquad \nu > \frac{2m}{2(\alpha-\gamma)-m}. \label{gammadef} 
    \end{equation}
    Alternatively, if $\beta$ is not an integer, we can simply set $\gamma=\beta$. In either case, we have $\|f\|_{C^\beta}\leq \|f\|_{C^\gamma_*}$ for all $f \in C^\gamma_*(\O_0)$.
    
	Now recall our assumption that the RKHS $\H$ of $\widetilde{\Pi}$ is continuously embedded into $H^\alpha(\O_0)$. We know from \cite[Theorem 3.1.2]{ET92} that the unit ball $U$ of this space satisfies
	$$\log \N(U, \|\cdot\|_{C^\gamma_*}, \ve) \leq \left(\frac{A}{\ve}\right)^{\tfrac{m}{(\alpha-\gamma)}} $$
	for some fixed $A >0$ and all $\ve >0$ small enough.
	Therefore, by \cite[Theorem 1.2]{LL99}, there exists $D >0$ such that for all $\ve >0$ small enough,
	$$\widetilde{\Pi}(\|c\|_{C^\beta}\leq \ve) \geq \widetilde{\Pi}(\|c\|_{C^\gamma_*}\leq \ve) \geq  \exp\left\{-D\ve^{-\tfrac{2m}{2(\alpha-\gamma)-m}}\right\}.$$
	Consequently, \eqref{pibn} implies that for $t>0$ large enough,
	\begin{align*}
		\Pi(\mathcal{B}_N(t)) &\geq  \frac{1}{\widetilde{\Pi}(\CMb)}\exp\left\{-\frac{\|c_0\|^2_{\H}}{2}-Dt^{\tfrac{2m}{2(\alpha-\gamma)-m}}\delta_N^{-\tfrac{2m}{2(\alpha-\gamma)-m}}\right\} \\
		&>   \frac{1}{\widetilde{\Pi}(\CMb)}\exp\left\{-\frac{\|c_0\|^2_{\H}}{2}-Dt^{\tfrac{2m}{2(\alpha-\gamma)-m}}\delta_N^{-\nu}\right\} \qquad \textrm{(by \eqref{nu} and \eqref{gammadef})}\\
		&=  \frac{1}{\widetilde{\Pi}(\CMb)}\exp\left\{-\frac{\|c_0\|^2_{\H}}{2}-Dt^{\tfrac{2m}{2(\alpha-\gamma)-m}}N\delta_N^2\right\} \\
		&\geq  \exp\{-C'N\delta_N^2\}
	\end{align*}
	for $C'= \log\left(\widetilde{\Pi}(\CMb)\right) + \frac{\|c_0\|^2_{\H}}{2}+Dt^{\tfrac{2m}{2(\alpha-\gamma)-m}}$. It now follows from Lemma \ref{smallball0} that for $t>0$ sufficiently large, there exists $C>0$ such that $\Pi(B_N) \geq \exp\{-CN\delta_N^2\}$. This completes the proof.
\end{proof}
Thus, we have verified Condition (1) of Theorem \ref{g-contr}. The next Lemma verifies Condition (2).

\begin{Lemma}\label{complexity}
	There exists $\widetilde{C} = \widetilde{C}(\O,\O_0,\gb,\beta,\ell)>0$ such that
	$$ \log \N(\CMb, h, \delta_N) \leq \widetilde{C}N\delta_N^2. $$
\end{Lemma}
\begin{proof}
	In order to construct a covering of $\CMb$, it suffices to construct such a covering of the $C^\beta_*(\O_0)$ - ball of radius $M$ centered at $0$. Therefore, if $U_\beta$ denotes the unit ball of $C^\beta_*(\O_0)$,
	\[
	\log \N(\CMb, \|\cdot \|_{L^2}, \delta_N) \leq \log \N(MU_\beta, \|\cdot\|_{L^2}, \delta_N). \]
	Now applying \cite[Theorem 3.1.2]{ET92} to the inclusion $C^\beta_*(\O_0) \hookrightarrow L^2(\O_0)$, we have
	$$\log \N(\CMb, \|\cdot \|_{L^2}, \delta_N) \leq \left(\frac{A'}{\delta_N}\right)^{\frac{m}{\beta}} $$
	for some $A' >0$. Since $\nu > m/\beta$, we get
	$$\log \N(\CMb, \|\cdot\|_{L^2},\delta_N) \leq b\delta_N^{-\nu} = b N\delta_N^2,$$
	where $b>0$. Now, Lemma \ref{hellinger-l2} and Corollary \ref{forward-c} imply that an $L^2$ ball of radius $\delta_N$ centered at any $c\in \CMb$ is contained in the Hellinger ball of radius $\frac{C_2'}{2\vol_{\gb}(\p\O)}\delta_N$ centered at the same point. Therefore, by suitably rescaling the constant $b$ to $\widetilde{C}(\O,\O_0,\gb,\beta,\ell,M)>0$, we get the desired complexity bound
	$$ \log \N(\CMb, h, \delta_N) \leq \widetilde{C}N\delta_N^2. $$
\end{proof}

\subsection{Posterior Convergence}

In this section, we will combine the results of Sections \ref{gencontr} and \ref{priorprop} to prove Theorem \ref{main}.

\begin{Theorem}\label{postcontr}
	Let $\Pi,\alpha, \beta, M, c_0$ be as in Theorem \ref{main}, $\nu,\delta_N$ as in \eqref{nu}, and $C>0$ as in Lemma \ref{smallball}. Then for $k'>0$ large enough, we have
	\begin{equation}\label{post1}
		P^N_{c_0}\left( \Pi(\{c \in \CMb : \|Z_c-Z_{c_0}\|_{L^2} \leq k'\delta_N\}|\mathcal{D}_N) \geq 1-e^{-(C+3)N\delta_N^2}\right) \to 1
	\end{equation}
	as $N \to \infty$. Moreover, for all $k''>0$ large enough,
	\begin{equation}\label{post2}
		P^N_{c_0}\left(\Pi(\{c \in \CMb : \|c-c_0\|_{L^2} \geq k''\delta_N^{1/2} \}|\mathcal{D}_N)\geq e^{-(C+3)N\delta_N^2}\right) \to 0
	\end{equation}
	as $N \to \infty$.
\end{Theorem}

\begin{proof}
	Combining Lemmas \ref{smallball} and \ref{complexity} with Theorem \ref{g-contr}, we get \eqref{post1} for all sufficiently large $k'>0$. To get \eqref{post2}, consider the event
	$$E_N = \{c \in \CMb : \|Z_c-Z_{c_0}\|_{L^2} \leq k'\delta_N \}.$$
	By Corollary \ref{inverse-c}, for any $c \in E_N$, 
	\begin{eqnarray*}
		\|c-c_0\|_{L^2} &\leq & C_1'\|Z_c-Z_{c_0}\|_{H^1} \\
		&\leq& C_1'\|Z_c-Z_{c_0}\|_{L^2}^{1/2}\|Z_c-Z_{c_0}\|_{H^{2}}^{1/2}
	\end{eqnarray*}
	by the standard interpolation result for Sobolev spaces. Therefore, by Theorem \ref{h2bounds},
	\[
	\|c-c_0\|_{L^2} \leq  C_1'(C_3')^{1/2}(k'\delta_N)^{1/2} \]
	Taking $k'' > C_1'(k'C_3')^{1/2}$, we conclude that
	$$ \|c-c_0\|_{L^2} \leq k''\delta_N^{1/2}. $$
	Combining this with \eqref{post1} gives us \eqref{post2}.
\end{proof}

The final step in the proof of Theorem \ref{main} is to prove that the posterior contraction rate in the above Theorem carries over to the posterior mean $\overline{c}_N = \E^\Pi[c|\D_N]$ as well. Let
$$0 < \omega < \frac{1}{2(2+\nu)}. $$
We note that $\omega$ can be made arbitrarily close to $1/4$ by choosing $\alpha, \beta$ appropriately. Indeed, if $\alpha$ and $\beta$ are sufficiently large, \eqref{nu} allows $\nu$ to be arbitrarily close to $0$. Correspondingly, $\omega$ can be made arbitrarily close to $1/4$. Next, define
$$ \omega_N := k''\delta_N^{1/2} = k''N^{-\frac{1}{2(2+\nu)}} = o(N^{-\omega}) $$
where $k''>0$ is as in Theorem \ref{postcontr}.

\begin{proof}[Proof of Theorem \ref{main}]
	Observe that
	\begin{eqnarray*}
		\|\overline{c}_N-c_0\|_{L^2} &=& \left\|\E^{\Pi}[c|\mathcal{D}_N]-c_0\right\|_{L^2} \\
		&\leq & \E^{\Pi}\left[\|c-c_0\|_{L^2}|\mathcal{D}_N\right] \quad \textrm{(by Jensen's inequality)} \\
		&\leq & \omega_N +\E^{\Pi}\left[\|c-c_0\|_{L^2}\mathds{1}_{\{\|c-c_0\|_{L^2}\geq \omega_N\}}\big|\mathcal{D}_N\right] \\
		&\leq & \omega_N +\E^{\Pi}\left[\|c-c_0\|_{L^2}^2|\mathcal{D}_N\right]^{1/2}\left[\Pi(\|c-c_0\|_{L^2}\geq \omega_N|\mathcal{D}_N)\right]^{1/2}
	\end{eqnarray*}
	by Cauchy-Schwarz inequality. Now it suffices to show that the second summand on the right hand side is stochastically $O(\omega_N)$ as $N \to \infty$.
	
	Arguing as in the proof of Theorem \ref{g-contr} and applying Lemma \ref{smallball}, we get that the events
	$$A_N' = \left\{ \int_{\CMb}\prod_{i=1}^N\frac{p_c}{p_{c_0}}(X_i,Y_i,Z_i)d\Pi(c) \geq e^{-(2+C)N\delta_N^2}\right\} $$
	satisfy $P^N_{c_0}(A_N')\to 1$ as $N \to \infty$. Here, $C$ is as in Lemma \ref{smallball}. Now, Theorem \ref{postcontr} implies
	\begin{align*}
		& P^N_{c_0}\left(\E^{\Pi}\left[\|c-c_0\|_{L^2}^2|\mathcal{D}_N\right]\times \Pi(\|c-c_0\|_{L^2}\geq \omega_N|\mathcal{D}_N) > \omega_N^2\right) \\
		& \leq P^N_{c_0}\left( \E^{\Pi}\left[\|c-c_0\|_{L^2}^2|\mathcal{D}_N\right]e^{-(C+3)N\delta_N^2}>\omega_N^2\right) +o(1),
	\end{align*}
	which is bounded above by
	\begin{align}
		& P^N_{c_0}\left( e^{-(C+3)N\delta_N^2}\E^{\Pi}\left[\|c-c_0\|_{L^2}^2|\mathcal{D}_N\right] > \omega_N^2 , A_N'\right) + o(1) \nonumber \\
		&= P^N_{c_0}\left(e^{-(C+3)N\delta_N^2}\frac{\int \|c-c_0\|_{L^2}^2\prod_{i=1}^N \frac{p_c}{p_{c_0}}(X_i,Y_i,Z_i)d\Pi(c)}{\int \prod_{i=1}^N \frac{p_c}{p_{c_0}}(X_i,Y_i,Z_i)d\Pi(c)} > \omega_N^2, A_N'\right) +o(1) \nonumber \\
		&\leq P^N_{c_0}\left( \int \|c-c_0\|_{L^2}^2 \prod_{i=1}^N\frac{p_c}{p_{c_0}}(X_i,Y_i,Z_i)d\Pi(c) > \omega_N^2e^{N\delta_N^2}\right) +o(1) \label{inter}
	\end{align}
	using the fact that $\int \prod_{i=1}^N\frac{p_c}{p_{c_0}}(X_i,Y_i,Z_i)d\Pi(c) \geq e^{-(C+2)N\delta_N^2}$ on  $A_N'$. Next, using Markov's inequality, \eqref{inter} can be further bounded above by
	\begin{align*}
		&\leq e^{-N\delta_N^2}\omega_N^{-2}\E^N_{c_0}\left[\int \|c-c_0\|_{L^2}^2\prod_{i=1}^N \frac{p_c}{p_{c_0}}(X_i,Y_i,Z_i)d\Pi(c)\right] +o(1) \\
		&= e^{-N\delta_N^2}\omega_N^{-2}\int \|c-c_0\|_{L^2}^2 \E_{c_0}^N\left[\prod_1^N\frac{p_c}{p_{c_0}}(X_i,Y_i,Z_i)\right]d\Pi(c) +o(1) \quad \textrm{(by Fubini's Theorem)} \\
		&\leq e^{-N\delta_N^2}\omega_N^{-2}\int \|c-c_0\|_{L^2}^2d\Pi(c) + o(1) \quad \left(\textrm{since } \E^N_{c_0}\left[\prod_1^N\frac{p_c}{p_{c_0}}\right]=1\right) \\
		&\lsim e^{-N\delta_N^2}\omega_N^{-2} +o(1) \lsim e^{-N\delta_N^2}N^{2\omega} +o(1) \to 0 \textrm{ as } N \to \infty
	\end{align*}
 This completes the proof.
\end{proof}

\bigskip

\noindent {\bf Acknowledgement:} HZ is partly supported by the NSF grant DMS-2109116.


\bibliographystyle{abbrv}
\bibliography{bibfile}

\end{document}